\documentclass[12pt]{scrartcl}
\usepackage[english]{babel}
\usepackage[utf8x]{inputenc}
\usepackage{amsfonts}
\usepackage{amsmath}
\usepackage{amssymb}
\usepackage{amsthm}
\usepackage{mathrsfs}
\usepackage[colorlinks=false,urlcolor=blue]{hyperref}
\usepackage{breakurl}
\usepackage{listings}
\usepackage{graphicx}
\usepackage{enumerate}
\usepackage{mathtools}
\usepackage{tensor}

\addtokomafont{section}{\rmfamily}
\addtokomafont{subsection}{\rmfamily}
\addtokomafont{paragraph}{\rmfamily}
\usepackage{parskip}

\swapnumbers
\theoremstyle{plain}
\newtheorem{thm}[paragraph]{Theorem}
\newtheorem{lem}[paragraph]{Lemma}
\newtheorem{kor}[paragraph]{Corollary}
\newtheorem{prop}[paragraph]{Proposition}
\newtheorem{qu}[paragraph]{Question}
\theoremstyle{definition}
\newtheorem{defn}[paragraph]{Definition}
\newtheorem{bem}[paragraph]{Remark}

\counterwithin{paragraph}{section}
\counterwithout{paragraph}{subsection}
\counterwithout{paragraph}{subsubsection}
\counterwithin{equation}{section}

\newcommand{\R}{\mathbb{R}}
\newcommand{\C}{\mathbb{C}}
\renewcommand{\H}{\mathbb{H}}
\newcommand{\CP}{\mathbb{CP}}
\newcommand{\Id}{\mathrm{Id}}
\newcommand{\Ric}{\mathrm{Ric}}
\newcommand{\scal}{\mathrm{scal}}
\newcommand{\vol}{\mathrm{vol}}
\newcommand{\Vol}{\operatorname{Vol}}
\newcommand{\Sym}{\operatorname{Sym}}
\newcommand{\tr}{\operatorname{tr}}
\newcommand{\End}{\operatorname{End}}
\newcommand{\Hom}{\operatorname{Hom}}
\newcommand{\Ad}{\operatorname{Ad}}
\newcommand{\ad}{\operatorname{ad}}
\newcommand{\Cas}{\mathrm{Cas}}
\newcommand{\Sy}{\mathscr{S}}
\newcommand{\Sf}{\mathfrak{S}}
\newcommand{\X}{\mathfrak{X}}
\newcommand{\U}{\operatorname{U}}
\newcommand{\SU}{\operatorname{SU}}
\newcommand{\su}{\mathfrak{su}}
\renewcommand{\u}{\mathfrak{u}}
\newcommand{\SO}{\operatorname{SO}}
\newcommand{\so}{\mathfrak{so}}
\newcommand{\Sp}{\operatorname{Sp}}
\renewcommand{\sp}{\mathfrak{sp}}
\newcommand{\Spin}{\operatorname{Spin}}
\newcommand{\GL}{\operatorname{GL}}
\newcommand{\gl}{\mathfrak{gl}}
\newcommand{\SL}{\operatorname{SL}}
\makeatletter
\@ifundefined{sl}{%
\newcommand{\sl}{\mathfrak{sl}}
}{
\renewcommand{\sl}{\mathfrak{sl}}
}
\makeatother
\renewcommand{\i}{\mathrm{i}}
\newcommand{\m}{\mathfrak{m}}
\newcommand{\h}{\mathfrak{h}}
\newcommand{\g}{\mathfrak{g}}
\renewcommand{\k}{\mathfrak{k}}
\renewcommand{\t}{\mathfrak{t}}
\newcommand{\TT}{\Sy^2_{\mathrm{tt}}}
\newcommand{\Isom}{\operatorname{Iso}}
\newcommand{\isom}{\mathfrak{iso}}
\newcommand{\Ltwoinprod}[2]{\mathchoice
  {\big(#1,#2\big)_{L^2}}%
  {(#1,#2)_{L^2}}%
  {(#1,#2)_{L^2}}%
  {(#1,#2)_{L^2}}%
}
\newcommand{\FN}[2]{[#1,#2]^{\mathrm{FN}}}
\newcommand{\quadric}{\mathscr{Q}}
\newcommand{\intprod}{\mathbin{\lrcorner}}
\newcommand{\Eop}{\mathrm{E}}
\newcommand{\rmE}{\mathrm{E}}
\newcommand{\fe}{\mathfrak{e}}
\newcommand{\rmF}{\mathrm{F}}
\newcommand{\ff}{\mathfrak{f}}
\newcommand{\rmS}{\mathrm{S}}
\newcommand{\Herm}{\mathfrak{H}}
\renewcommand{\AA}{\mathbb{A}}
\newcommand{\Univ}{\mathfrak{U}}
\newcommand{\Gr}{\mathrm{Gr}}
\newcommand{\Der}{\operatorname{\mathfrak{der}}}
\newcommand{\Str}{\operatorname{\mathfrak{str}}}
\renewcommand{\P}{\mathbb{P}}
\newcommand{\OO}{\mathbb{O}}
\newcommand{\SW}{\mathrm{S}}
\newcommand{\EH}{\mathcal{S}}
\newcommand{\CS}{\mathcal{C}}
\newcommand{\Diff}{\mathrm{Diff}}

\title{\rmfamily Sandwich operators and Einstein deformations of compact symmetric spaces related to Jordan algebras}
\author{Stuart James Hall\footnote{School of Mathematics, Statistics, and Physics,  Newcastle University, Newcastle upon Tyne, NE1 7RU, UK. Email: \texttt{stuart.hall@ncl.ac.uk}}, Paul Schwahn\footnote{Universidade Estadual de Campinas, IMECC, Rua Sérgio Buarque de Holanda 651, 13083-859 Campinas-SP, Brazil. Email: \texttt{schwahn@ime.unicamp.br}}, Uwe Semmelmann\footnote{Institut f\"ur Geometrie und Topologie, Fachbereich Mathematik, Universit\"at Stuttgart, Pfaffenwaldring 57, 70569 Stuttgart, Germany. Email: \texttt{uwe.semmelmann@mathematik.uni-stuttgart.de}}}

\begin{document}

\maketitle

\begin{abstract}
\noindent
We study the deformability of the symmetric Einstein metrics on the spaces $\SU(n)/\SO(n)$ and $\SU(2n)/\Sp(n)$, thereby concluding the problem to second order for all irreducible symmetric spaces. The obstruction integrals are calculated from invariant polynomials on certain Lie algebra representations. To aid the computation, we develop what we term \emph{sandwich operators} for compact Lie algebras and relate them to quadratic Casimir operators. We also explain the source of the infinitesimal Einstein deformations of irreducible symmetric spaces, except for the complex Grassmannians, by exploring their relation to simple Jordan algebras. Finally, we prove the nonlinear instability of most of the infinitesimally deformable irreducible compact symmetric spaces.

\medskip


\medskip

\end{abstract}

\section{Introduction}
\label{sec:intro}
\subsection{Overview and main results}

In the early 1980s, in the papers \cite{Koiso80}, \cite{Koiso82}, and \cite{Koiso83}, Koiso made a systematic study of the deformation theory of compact Einstein manifolds. He proved general structural results about the moduli space of Einstein metrics and applied this theory to provide a comprehensive account of deformability in the case of symmetric spaces. He showed that, in almost all cases, the canonical Einstein metric on an irreducible symmetric space does not admit deformations to first order and so the metric is isolated in the moduli space (isolated metrics are referred to as \emph{rigid}). Koiso found that there are five kinds of symmetric spaces that do admit infinitesimal Einstein deformations (see Proposition~\ref{symmetricied} for the precise statement); they are:
\[
\SU(n)/\SO(n), \  \SU(n), \ \SU(2n)/\Sp(n), \ \SU(p+q)/\mathrm{S}(\U(p)\times \U(q)), \ \mathrm{and} \ \mathrm{E}_{6}/\mathrm{F}_{4}.
\]
Koiso found an obstruction to integrating infinitesimal deformations to second order but left open the problem of calculating his obstruction on all but the space $\mathrm{E}_{6}/\mathrm{F}_{4}$. This problem has recently been successfully attacked for $\SU(n)$ \cite{BHMW} and for the complex Grassmannians  $\SU(p+q)/\mathrm{S}(\U(p)\times \U(q))$ \cite{OddGr}. In this article we investigate the second order deformations of the remaining two families -- $\SU({n)}/\SO(n)$ and $\SU(2n)/\Sp(n)$.

\begin{thm}
	\label{thm:maintheorem}
	Let $n\geq3$, and let $(M=G/K,g)$ be one of the Riemannian symmetric spaces $\SU(n)/\SO(n)$ or $\SU(2n)/\Sp(n)$.
	\begin{enumerate}[\upshape(i)]
		\item The set of infinitesimal Einstein deformations in $\varepsilon(g)$ which are integrable to second order corresponds to the variety
		\[\quadric=\left\{X\in\g\,\middle|\,X^2=\frac{\tr(X^2)}{m}I_m\right\}\]
		under the $G$-equivariant isomorphism $\varepsilon(g)\cong\g=\su(m)$, where $m=n$ or $m=2n$.
		\item If $n$ is odd, all infinitesimal Einstein deformations of $\SU(n)/\SO(n)$ are obstructed to second order, and hence the metric $g$ is rigid.
	\end{enumerate}
\end{thm}

By combining Theorem~\ref{thm:maintheorem}  with the work of Koiso \cite{Koiso82} and the results of the recent papers, \cite{BHMW} and \cite{OddGr}, we now have a complete understanding of integrability to second order of all of the deformable spaces on Koiso's list and thus a complete understanding of deformability to second order for \emph{all} compact irreducible symmetric spaces.

For the list of five infinitesimally deformable irreducible symmetric spaces, Koiso proved that the space of infinitesimal deformations $\varepsilon(g)$ is equivariantly isomorphic to the Lie algebra of Killing fields. If we write $M=G/K$, where $G$ is the isometry group, then $ \varepsilon(g) \cong \mathfrak{g}$, so working out Koiso's obstruction to integrability can be reduced to a calculation on the Lie algebra $\g$. The obstruction is closely related to various polynomials on $\g$, given by contracting certain products of traces, e.g.
\[-\sum_{i,j,k}\tr((X_iX_j+X_jX_i)X_k)^2\]
for an orthonormal basis $(X_i)$ of $\g$ (see Proposition~\ref{p0norm}), but one may also contract only over a subspace or subalgebra of $\g$ (see Section~\ref{sec:Section5}). Performing calculations with these polynomials involves the manipulation of certain Casimir-type objects which we term \emph{sandwich operators} and which we introduce in Section~\ref{sec:Section2}. These might be of independent interest for the simplification of similar expressions.

To provide a unified framework for dealing with both $\SU(n)/\SO(n)$ and $\SU(2n)/\Sp(n)$ it is useful to view them as symmetric spaces associated to \textit{formally real simple Jordan algebras}. In fact there is a very strong link between Jordan algebras and the infinitesimally deformable spaces appearing on Koiso's list, which we also hope is of independent interest; we elucidate this link in Section \ref{sec:Section3} and suggest topics for further investigation in this vein.

The question of whether an Einstein metric is infinitesimally deformable is related to the idea of stability for the Einstein--Hilbert action. As Einstein metrics are critical points for the action, a Einstein metric is said to be \textit{stable} if it is a local maximum when variations are restricted to fixed volume constant scalar curvature metrics. Thus stability can be investigated by considering the higher order derivatives of the action; these derivatives involve similar quantities to those found when considering infinitesimal deformations of Einstein metrics.  There is also a relationship between the Einstein--Hilbert stability of a metric and its stability as a fixed point of the Ricci flow. Such issues are discussed in Section \ref{sec:Section6}.

\subsection{Infinitesimal deformations of Einstein metrics and obstructions to integrability}

In this subsection, we give a little more detail regarding the important concepts we will use in the rest of the paper. For a comprehensive treatment of the deformation theory for Einstein metrics we refer the reader to \cite{besse} and \cite{NS23}.

Given a compact, oriented manifold $M^{n}$, the Einstein metrics (normalised to have unit volume) can be thought of as the zeroes of the Einstein operator
\[
\Eop(g) := \Ric(g)-\frac{\int_M\scal_g\,\vol_g}{n}g.
\]  
Infinitesimal deformations are in essence the critical points of this operator.

\begin{defn}
Let $(M^n,g)$ be an Einstein manifold, that is ${\mathrm{Ric}(g)=E g}$ with $E \in\R$. An \textit{infinitesimal Einstein deformation} (IED) is a section $h\in\Sy^2(M)=\Gamma(\Sym^2T^{\ast}M)$ satisfying the following conditions:
\begin{align}
\tr(h) &= 0, \label{eq:trace_free}\\
\mathrm{div}(h)&=0, \label{eq:div_free}\\
\nabla^\ast\nabla h -2\mathring{R}h &= 0, \label{eq:Lin_Ein}
\end{align}
where $\nabla^\ast\nabla$ is the connection Laplacian of the Levi-Civita connection, and $\mathring{R}$ is the \emph{curvature operator of the second kind} acting on symmetric 2-tensors. We denote the space of IED by $\varepsilon(g)$.
\end{defn}

The conditions \eqref{eq:trace_free} and \eqref{eq:div_free} correspond to eliminating conformal scaling and the action by the diffeomorphism group; the condition \eqref{eq:Lin_Ein} is then the equation $\Eop_g'(h)=0$. Tensors satisfying \eqref{eq:trace_free} and \eqref{eq:div_free} are known as \textit{transverse trace-free} and the space of such tensors is denoted $\TT(M)$. An IED $h\in\varepsilon(g)$ can be thought of as the first order approximation to a curve of Einstein metrics passing though $g$. Koiso \cite{Koiso82} proved that $h$ can be integrated to produce a second order approximation to a curve of Einstein metrics if and only if
\[
\Eop_{g}''(h,h) \perp_{L^{2}} \varepsilon(g),
\]
where we are using the $L^{2}$ inner product induced by the metric $g$. In particular, an obstruction to the integrability of $h\in \varepsilon(g)$ is the non-vanishing of the integral 
\[
\Psi(h) :=2\Ltwoinprod{\Eop_{g}''(h,h)}{h}.
\]
In his original treatment in \cite{Koiso82}, Koiso gave an expression for $\Psi$ in terms of local coordinates. Recently, Nagy and the third author \cite{NS23} have reinterpreted this expression in terms of the Frölicher--Nijenhuis bracket to provide an expression that is easier to work with (we refer the reader to Section \ref{sec:Section4} for more detail).

As mentioned previously, for the infinitesimally deformable symmetric spaces $G/K$, the space of deformations $\varepsilon(g)$ is $G$-equivariantly isomorphic to the Lie algebra $\g$. Koiso's obstruction can be seen as an element of $\Hom_{G}(\Sym^2\g,\g)$ and as this space is 1-dimensional for $G=\SU(n)$, $n\geq 3$, one really only needs to show that polynomial $\Psi$, which can be seen as an element of $\Sym^{3}\g^{\ast}$, does not vanish. The isomorphism between $\varepsilon(g)$ and $\mathfrak{su(n)}$ is described in Section \ref{sec:Section3} and then the calculation of the obstruction integral $\Psi$  is detailed in Sections \ref{sec:Section4} and \ref{sec:Section5}.

Let $M$ be a compact, orientable manifold. The Einstein--Hilbert action $\EH$ is the functional on the space of all Riemannian metrics on $M$ given by
\[
\EH(g) = \int_{M}\scal_g\,\vol_g.
\]
Restricted to the set of metrics of a fixed volume $c>0$, its critical points are precisely the Einstein metrics. It is known that, in general, Einstein metrics are never local maxima nor minima; however, an Einstein metric $g$ may be a local maximum when $\EH$ is further restricted to the set $\CS_{g}$ of constant scalar curvature metrics with the same volume as $g$. Such a metric is called \emph{$\EH$-stable} (or just \emph{stable}).

The stability of $g$ can be investigated by considering the second derivative of $\EH$ at $g$. Since $T_g\CS_g=\TT(M)\oplus T_g\Diff(M)g$, and the functional $\EH$ is invariant under diffeomorphisms, a sufficient condition for $\EH$-stability is \emph{linear stability}: namely
\[\EH_g''(h,h)<0\qquad\text{for all }h\in\TT(M).\]
However, an Einstein metric which is only \emph{linearly semistable} ($\EH_g''\leq0$ on $\TT(M)$) is not necessarily $\EH$-stable; within the null directions of $\EH_g''$ on $\TT(M)$, which are precisely $\varepsilon(g)$, higher order instabilities may hide.

The linear stability of irreducible symmetric spaces of compact type was settled in \cite{Koiso80,S22,SW22}. All of the spaces that admit IEDs (see Proposition~\ref{symmetricied}) turn out to be linearly semistable, but it was not clear whether they are also $\EH$-stable; they may permit higher order \emph{nonlinear} instabilities detectable only from higher variations of $\EH$. In this article we are able to demonstrate that the spaces $\SU(n)/\SO(n)$ and $\SU(2n)/\Sp(n)$ are unstable (Theorem~\ref{thm:Sunstable}).

There is another functional, also with motivations coming from geometric analysis, whose critical points include Einstein metrics of positive scalar curvature.  For a general compact Riemannian manifold $(M^n,g)$, Perelman's $\nu$-functional \cite{Per}, also called the \emph{shrinker entropy}, is given by
\[
\nu(g)  = \inf \left\{ \mathcal{W}(g,f,\tau)\,\middle|\, f\in C^{\infty}(M), \ \tau>0, \ (4\pi\tau)^{-\frac{n}{2}}\int_{M}e^{-f}\,\vol_g =  1 \right \},
\]
where
\[
\mathcal{W}(g,f,\tau) = \int_{M}\left(\tau (\mathrm{scal}_{g}+|\nabla f|^{2})+f-n\right)e^{-f}\,\vol_g.
\]
For any metric $g$, the infimum is always achieved by some pair $(f,\tau)$.  Perelman \cite{Per} proved that the critical points of $\nu$ are triples $(g,f,\tau)$ satisfying
\[
\Ric(g)+\nabla^{2}f=\frac{1}{2\tau}g.
\]
Such metrics are known as \emph{gradient Ricci solitons} and we can see they include Einstein metrics in the case that the function $f$ is constant (note here that the Einstein constant is $E=\frac{1}{2\tau}$). Perelman proved that the entropy is always increasing along the Ricci flow and thus perturbations of the metric which increase the entropy must necessarily be destabilizing. The second \cite{caohamilm,caozhu} and third \cite{kroenckeeinstein} variation of Perelman's entropy $\nu$ have been calculated, and they are closely related to the variations of the Einstein--Hilbert action. We end the article by discussing this relation in detail and giving a new proof of the instability of $\SU(n)/\SO(n)$ and $\SU(2n)/\Sp(n)$ as Ricci solitons.

\section{Operators associated to compact Lie algebras}
\label{sec:Section2}
\subsection{Casimir operators}

For any Lie algebra $\g$, denote with $\Univ\g$ its universal enveloping algebra. Elements of the center $Z(\Univ\g)$ are called \emph{Casimir elements}. Let $\pi: \g\to\gl(V)$ be a representation of $\g$, and denote its unique extension to a representation of $\Univ\g$ also with $\pi$. Then any Casimir element $C\in Z(\Univ\g)$ gives rise to an endomorphism $\pi(C)$ which intertwines with the representation of $\g$, i.e. $\pi(C)\in\End_\g V$. Such an endomorphism is called a \emph{Casimir operator} of $\g$ on $V$.

Suppose now that $\g$ is a compact Lie algebra, that is, the Lie algebra of a compact Lie group. Then there exists an $\ad(\g)$-invariant inner product on $\g$, which we fix (the existence of such an inner product is equivalent to $\g$ being compact). The inner product induces a particular Casimir element $C_2$ of degree two, often called \emph{the quadratic Casimir element} or just \emph{the} Casimir element. For any orthonormal basis $(X_i)$ of $\g$, it is given by
\[C_2=-\sum_iX_i^2.\]
In fact, if $\g$ is simple, $C_2$ is (up to scaling) the unique Casimir element of degree two. The associated operator (\emph{the quadratic Casimir operator} or just \emph{the} Casimir operator) on a representation $\pi: \g\to\gl(V)$ shall be denoted by
\[\Cas^\g_V=\pi(C_2)=-\sum_i\pi(X_i)^2.\]
More generally, we may define the quadratic Casimir element for any Lie algebra $\g$ which admits a \emph{nondegenerate} $\ad(\g)$-invariant bilinear form, for example if $\g$ is a \emph{reductive} Lie algebra. We are however primarily interested in the compact case, where the quadratic Casimir operator is nonnegative.

If the representation $V$ is irreducible, then Schur's Lemma implies that every Casimir operator on $V$ acts as multiplication by some constant. Assume in addition that $\g$ is semisimple, and let $\t\subset\g$ be a maximal torus. After choosing a set $R_+\subset\t$ of positive roots, we can associate to each irreducible representation of $\g$ its \emph{highest weight} $\lambda\in\t^\ast$ -- and conversely, each dominant integral weight $\lambda\in\t^\ast$ is the highest weight of a (up to equivalence) unique irreducible representation $V_\lambda$. Now a famous formula by Freudenthal tells us that the value of the quadratic Casimir constant on $V_\lambda$ is given by
\begin{equation}
 \Cas^\g_\lambda=\langle\lambda+2\rho_\g,\lambda\rangle,
 \label{eq:freudenthal}
\end{equation}
where $\rho_\g=\frac12\sum_{\alpha\in R_+}\alpha$ is the half-sum of positive roots, and the inner product on $\t^\ast$ is induced from that of $\g$.


\subsection{Invariant polynomials}

Take $\g$ to be a compact Lie algebra with fixed invariant inner product. We freely identify the symmetric algebra $\Sym\g^\ast$ with the algebra of polynomials $\R[\g]$ by
\[P(X):=P(X,\ldots,X),\qquad P\in\Sym\g^\ast,\ X\in\g,\]
and conversely by polarization.

If $G$ is a compact Lie group with Lie algebra $\g$, consider the subalgebra ${(\Sym\g^\ast)^G\cong\R[\g]^G}$ of $G$-invariant polynomials. Every homogeneous polynomial $P\in(\Sym^k\g^\ast)^G$ defines a Casimir element
\[\sum_{i_1,\ldots,i_k}P(X_{i_1},\ldots,X_{i_k})X_{i_1}\cdots X_{i_k}\in Z(\Univ\g),\]
where again $(X_i)$ denotes an orthonormal basis of $\g$. This prescription $\Sym\g^\ast\to Z(\Univ\g)$ is an isomorphism not of algebras, but of filtered vector spaces.

For simple Lie algebras $\g$, the polynomial algebras $\R[\g]^G$ are well understood; they can be  described as the free polynomial algebra on a set of homogeneous generators in certain degrees; the number of generators being equal to the rank of $\g$. For example for $G=\SU(n)$, the generators have degree $k=0,2,3,\ldots,n$ and are given by
\begin{equation}
X\mapsto\i^k\tr(X^k),\qquad X\in\g.
\label{eq:tracepowers}
\end{equation}

It will become necessary later to take inner products and norms of polynomials. If $(V,\langle\,\cdot\,,\cdot\,\rangle)$ is a Euclidean vector space, we shall take the inner product of $P,Q\in\Sym^kV^\ast$ to be
\begin{equation}
\label{eq:inprod}
\langle P,Q\rangle=\sum_{i_1,\ldots,i_k}P(X_{i_1},\ldots,X_{i_k})Q(X_{i_1},\ldots,X_{i_k}),
\end{equation}
where $(X_i)$ is an orthonormal basis of $V$. If the inner product on $V$ is invariant under the action of some group, then the natural action on $\Sym^kV^\ast$ will preserve its inner product as well.

\subsection{Sandwich operators}
\label{sec:sandwich}

We now define \emph{sandwich operators} which are of tremendous help when evaluating expressions involving contractions of trace forms and matrix products, the likes of which occur frequently in the calculations in \ref{sec:Section5}. Without naming them thus, sandwich operators have already been employed to prove the rigidity of the complex Grassmannians \cite[Def.~3.16]{OddGrArxiv}.

Let $\g$ be a a compact Lie algebra with an invariant inner product, $(X_{i})$ an orthonormal basis, and $\pi: \g\to\gl(V)$ a representation.\footnote{We shall write both $\gl(V)$ and $\End V$, depending on whether want to emphasize the Lie algebra structure or the associative algebra structure.}

\begin{defn}
 The \emph{sandwich operator} of $\g$ on $\End V$ is defined as
 \[\SW^\g_{\End V}: \End V\to\End V,\qquad A\mapsto\sum_i\pi(X_i)\,A\,\pi(X_i).\]
\end{defn}

Sandwich operators are closely related to quadratic Casimir operators:

\begin{lem}
For all $A\in\End V$, we have the identity
\[\SW^\g_{\End V}(A)=\frac{1}{2}(\Cas^\g_{\End V}(A)-\Cas^\g_V\circ A-A\circ\Cas^\g_V).\]
In particular, $\SW^\g_{\End V}\in\End_\g\End V$.
\end{lem}
\begin{proof}
Since $\g$ acts on $\End V$ by the commutator, we may calculate
\begin{align*}
 \Cas^\g_{\End V}(A)&=-\sum_i[\pi(X_i),[\pi(X_i),A]]\\
 &=-\sum_i(\pi(X_i)^2A-2\pi(X_i)A\pi(X_i)+A\pi(X_i)^2)\\
 &=2\SW^\g_{\End V}(A)+\Cas^\g_V\circ A+A\circ\Cas^\g_V.
\end{align*}
The intertwining property now follows from the fact that $\Cas^\g_{\End V}$ intertwines with $\g$, and one checks that post- or precomposition with $\Cas^\g_V$ does so as well.
\end{proof}

For irreducible representations $V$, the Casimir operator $\Cas^\g_V$ acts as multiplication by a constant. Hence we may abuse notation and write the following:
\begin{kor}
\label{sandwichonirr}
 If $V$ is irreducible, then
 \[\SW^\g_{\End V}=\frac{1}{2}\Cas^\g_{\End V}-\Cas^\g_V\cdot\Id.\]
\end{kor}

If $W\subset\End V$ is an irreducible subrepresentation and no confusion is possible, we may denote with $\SW^\g_W$ at the same time the restriction $\SW^\g_{\End V}\big|_W$ and its only eigenvalue, which Corollary~\ref{sandwichonirr} and Freudenthal's formula \eqref{eq:freudenthal} now allows us to evaluate. We will see several applications of this in the sequel.


\subsection{The cubic invariant of $\su(m)$}
\label{sec:suncubic}

We now let $G=\SU(m)$, and let the inner product on $\su(m)$ be the negative trace form,
\[\langle X,Y\rangle=-\tr(XY),\qquad X,Y\in\su(m).\]
As an example and as a preparation for what follows, let us discuss the invariant cubic form $P\in(\Sym^3\g^\ast)^G$, which is (up to a factor) uniquely given by
\begin{equation}
\label{eq:suncubic}
P(X)=2\i\tr(X^3),
\end{equation}
or in polarized form,
\[P(X,Y,Z)=\i\tr(XYZ+YXZ)=\i\tr(\{X,Y\}Z)\]
for $X,Y,Z\in\g$, where $\{X,Y\}=XY+YX$ is the \emph{anticommutator} of matrices.

As a first application of sandwich operators, we are able to calculate the norm of the invariant $P$. For this we need to record the Casimir and sandwich constants on the relevant representations.

\begin{lem}
\label{cassun}
 The quadratic Casimir constants of $\su(m)$ on the fundamental and adjoint representation are
 \[\Cas^\g_{\C^m}=\frac{m^2-1}{m}\quad \mathrm{and} \quad \Cas^\g_\g=2m.\]
\end{lem}
\begin{proof}
 It is well-known that the Killing form $B_{\su(m)}$ is related to the trace form by
 \begin{equation}
  B_{\su(m)}=2m\tr,
  \label{eq:killingformsu}
 \end{equation}
 and that the Casimir operator corresponding to $-B_{\su(m)}$ is the identity on the adjoint representation. Thus, for minus the trace form, we have $\Cas^\g_\g=2m$. Further, on the standard representation $\C^m$ we can take a trace
 \begin{align*}
  \tr(\Cas^\g_{\C^m})=-\sum_i\tr(X_i^2)=\dim\g=m^2-1,
 \end{align*}
 and by irreducibility it follows that $\Cas^\g_{\C^m}=\frac{m^2-1}{m}$.
\end{proof}


With Corollary~\ref{sandwichonirr}, we obtain:

\begin{lem}[\cite{haber}]
\label{swsun}
 For the representation $\g=\su(m)\subset\End\C^m$, the sandwich constant is given by \[
 \SW^\g_\g=\frac{1}{m}.
 \]
\end{lem}

\begin{prop}
\label{p0norm} Let $P$ be the invariant cubic in (\ref{eq:suncubic}). Then its norm is  
\[
 |P|^2=\frac{2(m^2-1)(m^2-4)}{m}.
\]
\end{prop}
\begin{proof}
First, we note that the anticommutator of Hermitian matrices is Hermitian, i.e.
\[\{\i\u(m),\i\u(m)\}\subseteq\i\u(m).\]
This implies that the orthogonal projection (under the trace form) to $\i\su(m)$ of an anticommutator is given by its trace-free part,
\[\{X,Y\}_{\i\su(m)}=\{X,Y\}_0-\frac{2}{m}\tr(XY)\Id.\]
Using that $\tr(X_iX_j)=-\delta_{ij}$, $\sum_i\tr(X_i^2)=-\dim\su(m)=-(m^2-1)$, and applying Lemmas~\ref{cassun} and \ref{swsun}, we find
\begin{align*}
 |P|^2&=-\sum_{i,j,k}\tr(\{X_i,X_j\}X_k)^2=2\sum_{i,j}\tr(X_iX_j\{X_i,X_j\}_0)\\
 &=2\sum_{i,j}\tr\left(X_iX_j\left(\{X_i,X_j\}+\frac{2}{m}\delta_{ij}\Id\right)\right)\\
 &=2\sum_{i,j}\tr(X_iX_jX_iX_j+X_iX_j^2X_i)+\frac{4}{m}\sum_i\tr(X_i^2)\\
 &=2\sum_i\tr(X_i\SW^\g_\g(X_i))-2\sum_i\tr(X_i\Cas^\g_{\C^m}X_i)+\frac{4}{m}\sum_i\tr(X_i^2)\\
 &=\left(\frac{2}{m}-2\frac{m^2-1}{m}+\frac{4}{m}\right)\sum_i\tr(X_i^2)\\
 &=\frac{2(m^2-1)(m^2-4)}{m}.
\end{align*}
\end{proof}

\begin{bem}
Related to $P$ is the cubic Casimir element
\[C_3=-\sum_{i,j,k}\tr(\{X_i,X_j\}X_k)X_iX_jX_k,\]
where again $(X_i)$ denotes an orthonormal basis. Letting $C_3$ act on the defining representation $\C^m$ and taking the trace, we find
\begin{align*}
 \tr_{\C^m}(C_3)&=-\frac{1}{2}\sum_{i,j,k}\tr(\{X_i,X_j\}X_k)^2=\frac{1}{2}|P|^2.
\end{align*}
Therefore, if $\mu_3$ denotes the eigenvalue of the Casimir operator associated to $C_3$ on $\C^m$, we have
\[|P|^2=2m\mu_3.\]
In \cite[p.~5]{haber}, Haber had already worked out\footnote{In his notation, we may take $T_a=(\i/\sqrt{2})X_a$, thus $C_{3F}=\mu_3/4$.} $\mu_3$ to be
\[\mu_3=\frac{(m^2-1)(m^2-4)}{m^2}.\]
The calculation there already relies fundamentally on the sandwich operator $\SW^\g_\g$ \cite[(13)]{haber}. The same method may be used to work out the eigenvalues of higher Casimir operators on the fundamental representation, as long as they come from invariant polynomials of the form \eqref{eq:tracepowers}.
\end{bem}

\section{Symmetric spaces and Jordan algebras}
\label{sec:Section3}
We recall Koiso's list of irreducible symmetric spaces of compact type $M=G/K$ admitting infinitesimal Einstein deformations.

\begin{prop}[\cite{Koiso80}, Thm.~1.1]
\label{symmetricied}
 Of the irreducible symmetric spaces of compact type, only the following are infinitesimally deformable:
  \begin{enumerate}[\upshape(i)]
  \item $\SU(n)/\SO(n)$, $n\geq3$,
  \item $\SU(n)=\frac{\SU(n)\times\SU(n)}{\SU(n)}$, $n\geq 3$,
  \item $\SU(2n)/\Sp(n)$, $n\geq3$,
  \item $\rmE_6/\rmF_4$,
  \item the complex Grassmannians $\Gr_p(\C^{p+q})$, $q\geq p\geq 2$.
  \end{enumerate}
\end{prop}

In \cite[Prop.~2.1]{GGiso} it was observed that spaces (i)--(iv) have in common the existence of a $G$-invariant (hence parallel) symmetric $3$-tensor $\sigma\in\Sy^3(M)$. This symmetric $3$-tensor furnishes an equivariant isomorphism between the $\varepsilon(g)$ and the space of Killing vector fields:

\begin{prop}[Gasqui--Goldschmidt, cf.~Lem.~2.8 in \cite{BHMW}]
\label{iedparam}
 On any of the spaces {\upshape(i)--(iv)}, the map
 \[\isom(M,g)\longrightarrow\varepsilon(g),\qquad Z\mapsto Z\intprod\sigma\]
 is an isomorphism. In particular $\varepsilon(g)\cong\g$.
\end{prop}

As we will see in~\ref{subsec:symspa}, the existence of $\sigma$ on the spaces (i)--(iv) is not a coincidence, but related to the four normed division algebras.

The spaces (v) are discussed in \cite{Hall24,OddGr,OddGrArxiv}, see also \cite{NS23} for $p=2$. In this case it is also true that $\varepsilon(g)\cong\isom(M,g)$, although the parametrization of infinitesimal Einstein deformations by Killing fields is not through a bundle map (order zero), but rather given by a certain (first order) differential operator.

\subsection{Jordan algebras and associated Lie algebras}

We give a brief overview of the relevant structures and relations between them. Pieces of the following discussion and some pertinent background material can be found in \cite[\S3--4]{baez}, \cite{braunkoecher}, \cite[\S II--III]{cones}, \cite{jacobson}, \cite{koecher}, \cite[\S VIII]{LoosII}, \cite[\S I.8]{satake}, \cite[\S II, \S IV]{schafer}.

\paragraph{Jordan algebras.}

For a (real) normed division algebra $\AA$, let $\Herm(n,\AA)$ denote the space of Hermitian $n\times n$-matrices with entries in $\AA$, and let $\Herm_0(n,\AA)$ denote the subspace of traceless matrices. Declare the \emph{Jordan product} of $\Herm(n,\AA)$ to be
\[A\circ B=\frac12\{A,B\}.\]
For associative $\AA$ (i.e.~$\AA=\R,\C,\H$), these becomes instances of a \emph{Jordan algebra}, meaning a commutative, not necessarily associative algebra $(J,\circ)$ over $\R$ satisfying
\[(x\circ y)\circ (x\circ x)=x\circ(y\circ(x\circ x)),\qquad x,y\in J.\]
A Jordan algebra is called \emph{formally real} if the vanishing of a sum of squares implies the vanishing of individual terms, i.e.~for any $x_1,\ldots,x_n\in J$,
\[x_1^2+\ldots+x_n^2=0\qquad\text{implies}\qquad x_1=\ldots=x_n=0.\]
Formally real Jordan algebras always decompose as a direct sum of \emph{simple} ones, which classified in \cite{jordanclassif}. The algebras $(\Herm(n,\AA),\circ)$ appear as families in this classification, and are sometimes called \emph{special Jordan algebras}. However, the classification also contains one exceptional entry: the \emph{Albert algebra} $\Herm(3,\OO)$, which is a Jordan algebra despite $\OO$ being nonassociative.

\paragraph{Derivation and structure algebra.}

To any Jordan algebra $J$ we may associate its Lie algebra of derivations $\Der J$. One also defines the \emph{structure algebra} (or \emph{Lie multiplication algebra}) $\Str J$, which is the Lie subalgebra of $\gl(J)$ generated by the left-/right-multiplication operators of $J$. If $J$ is a formally real, unital Jordan algebra, we have (\cite[(I.7.19)]{satake}, cf.~Thm.~2.5 and \S IV.1 in \cite{schafer})
\[\Str J=\Der J\oplus J,\]
 where the Lie bracket on $J$ is given by
\[J\times J\to\Der J,\qquad[x,y]z:=x\circ(y\circ z)-y\circ(x\circ z).\]
Thus $(\Str J,\Der J)$ is a symmetric pair of Lie algebras. The associated (noncompact) symmetric space $D_J$ is called a \emph{homogeneous domain of positivity}, and it can be realized as a connected component of the set of invertible elements in $J$ \cite[Thm.~VIII.2.2]{LoosII}.

Formally real Jordan algebras possess a distinguished inner product \cite[\S III]{cones}, called the \emph{Killing form}, which is invariant in the sense that
\[\langle x\circ y,z\rangle=\langle x,y\circ z\rangle,\qquad x,y,z\in J.\]
This implies that the Lie algebra $\Der J\oplus\i J\subset(\Str J)^\C$ preserves the inner product, and hence is a compact Lie algebra. Thus $\Str J$ is reductive, and its semisimple part is $\Str_0J:=[\Str J,\Str J]$. We shall be particularly interested in the symmetric spaces associated to the pairs $(\Str_0 J,\Der J)$, or rather their compact duals.

On every finite-dimensional, unital Jordan algebra $J$ (more generally, on any power-associative unital algebra) there exists a canonical polynomial $\det: J\to\R$ called the \emph{determinant}, which is the polynomial of minimal degree with the property that any $x\in J$ is annihilated by
\[\det(\lambda1_J-x)\in\R[\lambda],\]
and satisfying $\det(1_J)=1$ \cite[\S II.2]{cones}.
Taking minus the coefficient of the second highest power of $\lambda$ in the above polynomial, one obtains a linear form $\tr: J\to\R$ called the \emph{trace}. If $J$ is a simple formally real Jordan algebra, $\Str_0J$ coincides with the annihilator of the determinant in $\gl(J)$, and moreover $\Str_0J=\Der J\oplus J_0$, where $J_0$ is the kernel of the trace on $J$ \cite[Satz~III.5.5 and \S IX.5]{braunkoecher}.

\subsection{The emerging symmetric spaces}


\paragraph{Special and exceptional Jordan algebras.}

We now let $J=\Herm(n,\AA)$, where either $\AA=\R,\C,\H$ and $n\geq3$, or $\AA=\OO$ and $n=3$. Then it is well-known \cite{baez} that
\[\Der J=\su(n,\AA):=\Der\AA\oplus\{X\in\AA^{n\times n}\,|\,\bar{X}^\top=-X,\ \tr X=0\}.\]
The semisimple part of the structure algebra is then given by
\[\Str_0J=\su(n,\AA)\oplus\Herm_0(n,\AA)=\sl(n,\AA):=\Der\AA\oplus\{X\in\AA^{n\times n}\,|\,\tr X=0\},\]
where $\Herm_0(n,\AA)=J_0$ is the subspace of traceless matrices. The compact real form of $\Str_0J$ is in turn given by
\[\g:=\su(n,\AA)\oplus\i\Herm_0(n,\AA).\]
Since $\Str_0J$ is the Lie algebra of endomorphisms of $J$ annhilating the determinant, we may view its compact form $\g$ as the Lie algebra of complex-linear endomorphisms of $J^\C$ preserving both the determinant and the Killing form.

The concrete Lie algebras obtained in each case are listed in Table~\ref{algebras} (see \cite[Thm.~10]{orlitzky} for the corresponding automorphism groups).

\begin{table}[th]
\centering
\renewcommand{\arraystretch}{1.25}
 \begin{tabular}{c|c|c|c}\hline
  $\AA$&$\Der J=\su(n,\AA)$&$\Str_0 J=\sl(n,\AA)$&compact form $\g$\\\hline\hline
  $\R$ ($n\geq3$)&$\so(n)$&$\sl(n,\R)$&$\su(n)$\\\hline
  $\C$ ($n\geq3$)&$\su(n)$&$\sl(n,\C)$&$\su(n)\oplus\su(n)$\\\hline
  $\H$ ($n\geq3$)&$\sp(n)$&$\sl(n,\H)$&$\su(2n)$\\\hline
  $\OO$ ($n=3$)&$\ff_4$&$\fe_6^{-26}$&$\fe_6$\\\hline
 \end{tabular}
\caption{Lie algebras associated to the Jordan algebras $J=\Herm(n,\AA)$.}
\label{algebras}
\end{table}

\paragraph{Compact irreducible symmetric spaces.} \label{subsec:symspa}

For $J$ as above, let $\k:=\Der J$ and $\m:=\i J_0$. Looking at Table~\ref{algebras}, we see that the symmetric pairs $(\g,\k)$ correspond to the compact symmetric spaces $M=G/K$ in Proposition~\ref{symmetricied} (i)--(iv). Among all irreducible symmetric spaces of compact type, these are the only ones carrying an invariant (and thus paralllel) symmetric $3$-tensor; moreover, in each of those cases, the space of such tensors $\Sy^3(M)^G\cong(\Sym^3\m^\ast)^K$ is one-dimensional \cite[Prop.~2.1]{GGiso}. The source of this tensor $\sigma$ is now easily explained: it comes from the trilinear form corresponding to Jordan multiplication (respectively the anticommutator of matrices), composed with projection to the trace-free part:
\begin{equation}
\sigma(X,Y,Z):=\i\tr(\{X,Y\}Z)=g(-\i\{X,Y\}_0,Z),
\label{sigma}
\end{equation}
where $X,Y,Z\in\m=\i\Herm_0(n,\AA)\cong T_oM$, if we choose the Riemannian metric to be induced by minus the trace form on $\m\subset\g$. As a consequence of this and Proposition~\ref{iedparam}, the IED of $(M,g)$ can be thought of as the symmetric endomorphisms of $TM$ corresponding to ``Jordan multiplication by a Killing field''.

\begin{bem}
 Besides the special Jordan algebras $\Herm(n,\AA)$ for $n\geq3$, $\AA=\R,\C,\H$ and the exceptional Jordan algebra $\Herm(3,\OO)$, the other entries in the classification \cite{jordanclassif} are the \emph{spin factors} $\Sf_n:=\R^n\times\R$, $n\geq0$, with Jordan product given by
 \[(u,\alpha)\circ(v,\beta)=(\alpha v+\beta u,\langle u,v\rangle+\alpha\beta).\]
 Their derivation algebra is $\Der\Sf_n=\so(n)$, acting on the $\R^n$ factor in the standard way \cite[Thm.~10]{orlitzky}. One may show that the semisimple part of its structure algebra is
 \[\Str_0\Sf_n=\so(n)\oplus\R^n\cong\so(n,1),\]
 hence the symmetric spaces related to $(\Str_0\Sf_n,\Der\Sf_n)$ are real hyperbolic space resp.~the sphere $S^n$. Because the $2\times2$-special Jordan algebras are isomorphic to spin factors,
 \[\Herm(2,\AA)\cong\Sf_{a+1},\qquad a=\dim\AA,\]
 this also furnishes the isomorphisms $\so(a+1,1)\cong\sl(2,\AA)$.
\end{bem}

\paragraph{Projective spaces from Jordan algebras.}
\label{jordanproj}


The projective spaces over the four normed division algebras may be exhibited as certain subsets of formally real Jordan algebras, an observation which in this generality is due to U.~Hirzebruch \cite{crosses}, and which has also been used to define the geometry of the Cayley plane $\OO\P^2$ \cite{freudenthal,jordan}. In analogy to projectors from linear algebra, one may in any Jordan algebra $J$ consider the subsets of idempotents which have trace one or two,
\begin{align*}
 \P_J:=\{x\in J\,|\,x^2=x,\ \tr x=1\},\\
 \mathbb{L}_J:=\{x\in J\,|\,x^2=x,\ \tr x=2\}.
\end{align*}
We may interpret $\P_J$ as a set of \emph{points}, $\mathbb{L}_J$ as a set of \emph{lines}, and call $(p,\ell)\in\P_J\times \mathbb{L}_J$ \emph{incident} if $p\circ\ell=p$. Starting from the Jordan algebras $J=\Herm(n,\AA)$, the corresponding incidence geometries are the projective spaces $\P_J=\AA\P^{n-1}\subset\Herm(n,\AA)$.

It is sometimes convenient to identify $\P_J$ with its image in the projectivization $\P(J)$ of the vector space $J$; this map is one-to-one because every element of the image clearly has a unique representative of trace one.

Suppose that $J$ is simple and formally real. Let $G'\subset\GL(J)$ be the connected Lie group with Lie algebra $\Str_0J$. Under the above identification, $\P_J$ acquires an action of $G'$; indeed, $G'$ acts on $\P_J$ by \emph{collineations}, i.e.~it preserves the incidence relation. This can be verified case by case for the known classification \cite{jordanclassif}; a direct proof was given by Freudenthal for the exceptional Jordan algebra $J=\Herm(3,\OO)$ \cite{freudenthal}, but seems to be lacking for the general case. Moreover, it is not clear what sort of incidence geometries one obtains when relaxing the assumptions on $J$.

\begin{qu}
 Is there a direct proof that $G'$ acts on $\P_J$ transitively and by collineations?
\end{qu}

\begin{qu}
 Starting from any (commutative) power-associative algebra $A$, what incidence geometries $\P_A$ can be obtained? Under which conditions is $\P_A$
 \begin{itemize}
  \item a projective space?
  \item a projective plane?
  \item Desarguesian/non-Desarguesian?
 \end{itemize}
\end{qu}

\paragraph{Complexification and Lagrangians of projective spaces.}

We may now follow the construction outlined in \cite{complexify} to define a ``complexification'' of the projective space $\P_J$. First, we complexify the Lie group $G'$ to obtain an action of a complex Lie group $G^\C$ on $\P(J^\C)$. We now take
\[\P_J^\C:=G^\C.\P_J\subset\P(J^\C)\]
to be the $G^\C$-orbit of $\P_J\subset\P(J)\subset\P(J^\C)$. Since $\P_J^\C$ is compact, it is in fact the equal to $G.\P_J$, where $G\subset G^\C$ is the maximal compact subgroup, i.e.~compact real form of $G^\C$. The resulting spaces for the Jordan algebras $J=\Herm(n,\AA)$ are listed in Table~\ref{projspaces}.

\begin{table}[th]
\centering
\renewcommand{\arraystretch}{1.6}
 \begin{tabular}{c|c|c}\hline
  $\AA$&$\P_J=G'/H'$&$\P_J^\C=G/H$\\\hline\hline
  $\R$ ($n\geq3$)&$\R\P^{n-1}=\frac{\SL(n,\R)}{\rmS(\GL(n-1,\R)\times\GL(1,\R))}$&$\C\P^{n-1}=\frac{\SU(n)}{\rmS(\U(n-1)\times\U(1))}$\\\hline
  $\C$ ($n\geq3$)&$\C\P^{n-1}=\frac{\SL(n,\C)}{\rmS(\GL(n-1,\C)\times\GL(1,\C))}$&$\C\P^{n-1}\times\C\P^{n-1}=\frac{\SU(n)^2}{\rmS(\U(n-1)\times\U(1))^2}$\\\hline
  $\H$ ($n\geq3$)&$\H\P^{n-1}=\frac{\SL(n,\H)}{\rmS(\GL(n-1,\H)\times\GL(1,\H))}$&$\Gr_2(\C^{2n})=\frac{\SU(2n)}{\rmS(\U(2)\U(2n-2))}$\\\hline
  $\OO$ ($n=3$)&$\OO\P^2=\frac{\rmE_6^{-26}}{\Spin(9,1)}$&$(\OO\otimes\C)\P^2=\frac{\rmE_6}{\Spin(10)\U(1)}$\\\hline
 \end{tabular}
\caption{Projective spaces inside $J=\Herm(n,\AA)$ and their complexification.}
\label{projspaces}
\end{table}

Let $K$ be the maximal compact subgroup of $G$, whose Lie algebra is $\Der J$. In the cases $J=\Herm(n,\AA)$, the compact irreducible symmetric space $M=G/K$ has the additional interpretation as a \emph{Lagrangian Grassmannian} of Lagrangian $\AA^n$-subspaces in $(\AA\otimes\C)^n$ (see \cite{huang} for a precise definition and proof). Equivalently, we may view $M$ as the space of ``linear, Lagrangian'' inclusions $\P_J\subset\P_J^\C$ \cite[\S10.K, Table~2]{besse}. Indeed, the isotropy group $K$ is precisely the identity component of $\Isom(\P_J)$.

\begin{bem}
 Carrying out the construction in~\ref{jordanproj} with the spin factors $\Sf_n$, one obtains the \emph{celestial sphere}
 \[\P_{\Sf_n}=\{(u,\alpha)\in\R^n\times\R\,|\,|u|^2=\alpha^2=\tfrac12\}\cong S^{n-1},\]
 where all points lie on a single line \cite[\S3.3]{baez}; by the isomorphism $\Herm(2,\AA)\cong\Sf_{a+1}$, this also includes the projective lines $\AA\P^1$. Identifying $\P_{\Sf_n}$ with the space of light rays through the origin in $\R^{n,1}$ gives the (transitive) action of the group $G'=\SO(n,1)^0$. The complexification procedure then yields the complex quadric
 \[\{z_0^2+\ldots+z_n^2=0\}\subset\C\P^n\]
 as a homogeneous space of $G=\SO(n+1)$, which is isomorphic to the Grassmannian of oriented planes $\widetilde{\Gr}_2(\R^{n+1})=\SO(n+1)/(\SO(2)\times\SO(n-1))$.
\end{bem}

\begin{bem}
 For completeness we note that there is yet another construction of symmetric spaces from formally real Jordan algebras: if $D_J$ denotes the domain of positivity associated to a Jordan algebra $J$, then the \emph{half-space} belonging to $J$ is defined as the open subset
 \[H_J:=J+\i D_J\subset J^\C.\]
 It is a Hermitian symmetric space of noncompact type \cite[Thm.~VIII.2.5]{LoosII}, cf.~\cite{halfspaces}. This gives a way to define the exceptional Lie group $\rmE_7^{-25}$, namely as the isometry group of the half-space $H_J$ associated to $J=\Herm(3,\OO)$. Even more constructions of symmetric spaces from Jordan algebras are discussed in \cite{helwig}.
\end{bem}

\section{Integrability to second order}
\label{sec:Section4}
\subsection{Koiso's obstruction integral}

Let $(M,g)$ be a compact, oriented Einstein manifold, with Einstein constant $E$. In \cite{Koiso82}, Koiso gives an equivalent criterion for an IED to be integrable to second order:

\begin{prop}[\cite{Koiso82}, Lem.~4.7]
\label{int2iff}
 An IED $h\in\varepsilon(g)$ is integrable to second order if and only if
 \[\Eop_g''(h,h)\perp_{L^2}\varepsilon(g).\]
\end{prop}

Define \emph{Koiso's obstruction integral} $\Psi$ as the polynomial on $\varepsilon(g)$ given by
\[\Psi(h):=2\Ltwoinprod{\Eop_g''(h,h)}{h},\qquad h\in\varepsilon(g).\]
Clearly, a necessary condition for $h\in\varepsilon(g)$ to be integrable to second order is that $\Psi(h)=0$. A formula for $\Psi$ was worked out in \cite[Lem.~4.3]{Koiso82}, and simplified in \cite[Prop.~3.14]{NS23} in terms of the Frölicher--Nijenhuis bracket.

Using the metric, we may identify any covariant $2$-tensor with a section of $\End TM$, or equivalently, a vector-valued one-form. The Frölicher--Nijenhuis bracket on vector valued one-forms is the symmetric bilinear map
\[\FN{\,\cdot\,}{\cdot\,}: \Omega^1(M,TM)\times\Omega^1(M,TM)\longrightarrow\Omega^2(M,TM)\]
determined by
\begin{align*}
 \FN{h}{h}(X,Y)&=-h^2[X,Y]+h([hX,Y]+[X,hY])-[hX,hY]\\
 &=-(\nabla_{hX}h)Y+(\nabla_{hY}h)X+h(d^\nabla h(X,Y))
\end{align*}
for $h\in\Omega^1(M,TM)$. With that, Koiso's obstruction integral $\Psi$ can be written as
\begin{align}
\Psi(h)&=\int_M\left(3\langle\delta\FN{h}{h},h\rangle-E\tr(h^3)\right)\vol_{g}\notag\\\
&=\int_M\left(3\langle\FN{h}{h},d^\nabla h\rangle-E\tr(h^3)\right)\vol_{g}.\label{eq:psiexpr}
\end{align}
In the last line we have used that $\delta$ is formally $L^2$-adjoint to the differential operator
\[d^\nabla:\quad\Omega^1(M,TM)\to\Omega^2(M,TM),\qquad d^\nabla\alpha=\sum_ie^i\wedge\nabla_{e_i}\alpha,\]
given that we use the inner products
\begin{align*}
\langle\alpha,\beta\rangle&=\sum_i\langle\alpha(e_i),\beta(e_i)\rangle,&\alpha,\beta&\in\Omega^1(M,TM),\\
\langle\alpha,\beta\rangle&=\frac12\sum_{i,j}\langle\alpha(e_i,e_j),\beta(e_i,e_j)\rangle,&\alpha,\beta&\in\Omega^2(M,TM).
\end{align*}
The space $\varepsilon(g)$ is always a (finite-dimensional) representation of the isometry group, and the obstruction integral $\Psi$ is an invariant cubic form on this representation.

In \cite{NS23,OddGrArxiv} it is pointed out that the trilinear form
\[\varepsilon(g)^3\to\R,\qquad (h_1,h_2,h_3)\mapsto\Ltwoinprod{\Eop_g''(h_1,h_2)}{h_3}\]
is completely symmetric. Thus all the information about integrability to second order is already contained in the cubic form $\Psi$. Indeed, we can reformulate the condition in Proposition~\ref{int2iff} as follows:

\begin{prop}
\label{critical}
 An infinitesimal Einstein deformation $h\in\varepsilon(g)$ is integrable to second order if and only if it is a critical point of the polynomial $\Psi$; that is, if
 \[\Psi(h,h,k)=0\qquad\text{for all }k\in\varepsilon(g).\]
\end{prop}

\subsection{The second order obstruction on symmetric spaces}

Let us from now on assume that $(M=G/K,g)$ is one of the spaces in Proposition~\ref{symmetricied} (i)--(iv), and let $\sigma\in\Sy^3(M)$ be the parallel symmetric $3$-tensor defined in \eqref{sigma}. In light of Proposition~\ref{iedparam}, we would like to express Koiso's obstruction in terms of Killing vector fields.

Let $\g=\k\oplus\m$ be the Cartan decomposition; as usual, we identify $\m\cong T_oM$ with the tangent space at the identity coset $o\in M$. Furthermore, denote with $Z_\k,Z_\m$ the respective projections of $Z\in \g$ to $\k,\m$.

Recall that on an irreducible symmetric space, the Killing vector fields are \emph{fundamental vector fields}, that is, the map
\[\g\to\isom(M,g),\quad Z\mapsto\tilde Z\quad\text{where}\qquad\tilde Z_p:=\frac{d}{dt}\big|_{t=0}\exp(tX).p,\quad p\in M,\]
is an isomorphism. Starting from a Killing field $\tilde Z$, the Lie algebra element $Z\in\g$ can be recovered as follows:
\begin{equation}
 \label{eq:killingcorresp}
 Z_\m=\tilde Z_o,\qquad\ad(Z_\k)\big|_\m=(\nabla\tilde Z)_o=\frac{1}{2}(d\tilde Z)_o,
\end{equation}
cf.~\cite[Prop.~10.1.4]{petersen}. With the identification $\g\cong\isom(M,g)$, we may view Koiso's obstruction as an invariant cubic form on $\g$, that is, $\Psi\in(\Sym^3\g^\ast)^G$. It remains to determine an expression for it in terms of Lie-algebraic data.

Let $h_Z:=\tilde Z\intprod\sigma\in\varepsilon(g)$ denote the IED associated to $Z\in\g$. Using $\nabla\tilde Z=\frac12d\tilde Z$ and $\nabla\sigma=0$, direct calculation leads to the identities
\begin{align}
 2d^\nabla h_Z(X,Y)&=\sigma(d\tilde Z(X),Y)-\sigma(d\tilde Z(Y),X),\label{eq:dnabla}\\
 2\FN{h_Z}{h_Z}(X,Y)&=-\sigma(d\tilde Z(\sigma(Z,X)),Y)+\sigma(dZ(\sigma(\tilde Z,Y)),X)\notag\\
 &\quad+\sigma(\tilde Z,\sigma(d\tilde Z(X),Y))-\sigma(\tilde Z,\sigma(d\tilde Z(Y),X)).\label{eq:FNbracket}
\end{align}

We are now ready to express $\Psi(h_Z)$ as a polynomial integral in $Z$. Equip $G$ with a Haar measure such that $\int_Gdx=\int_M\vol_{g}$.

\begin{lem}
\label{psipol}
 Let the polynomials $Q,R,R_1,\ldots, R_4\in(\Sym^3\g^\ast)^K$ be defined by
 \begin{align*}
  Q(Z)&=\sum_{i,j,k}\sigma(Z_\m,E_i,E_j)\sigma(Z_\m,E_i,E_k)\sigma(Z_\m,E_j,E_k),\\
  R_1(Z)&=\sum_{i,j,k,l}\sigma(Z_\m,E_i,E_j)\sigma([Z_\k,E_i],E_k,E_l)\sigma([Z_\k,E_k],E_j,E_l),\\
  R_2(Z)&=\sum_{i,j,k,l}\sigma(Z_\m,E_i,E_j)\sigma([Z_\k,E_i],E_k,E_l)\sigma([Z_\k,E_j],E_k,E_l),
 \end{align*}
 \begin{align*}
  R_3(Z)&=\sum_{i,j,k,l}\sigma(Z_\m,E_i,E_j)\sigma(E_i,E_k,E_l)\sigma(E_j,E_k,[Z_\k,[Z_\k,E_l]]),\\
  R_4(Z)&=\sum_{i,j,k,l}\sigma(Z_\m,E_i,E_j)\sigma(E_i,E_k,E_l)\sigma(E_j,[Z_\k,E_k],[Z_\k,E_l]),\\
  R&=2R_1-2R_2-2R_3+2R_4,
 \end{align*}
 where $(E_i)$ is an orthonormal basis of $\m$. Then the obstruction polynomial $\Psi\in(\Sym^3\g^\ast)^G$ is given by
 \[\Psi(h_Z)=\int_G\left(3R(\Ad(x^{-1})Z)-EQ(\Ad(x^{-1})Z)\right)dx,\qquad Z\in\g.\]
\end{lem}
\begin{proof}
 Any function $f\in C^\infty(M)$ gives rise to a $K$-invariant function $\hat f\in C^\infty(G)^K$ by setting $\hat f(x)=f(x.o)$, and we have
 \[\int_Mf\,\vol_{g}=\int_G\hat f(x)\,dx.\]
 More generally, any section of a homogeneous vector bundle $s\in\Gamma(G\times_KV)$ corresponds to a $K$-equivariant function $\hat s\in C^\infty(G,V)^K$ by $s(x.o)=[x,\hat s(x)]$. From \eqref{eq:psiexpr}, \eqref{eq:dnabla} and \eqref{eq:FNbracket}, we see that the integrand in $\Psi(h_Z)$ is of the form $I(\tilde Z,d\tilde Z)$ for a certain section $I\in\Gamma(\Sym^3(TM\oplus\Lambda^2TM)^\ast)$. Thus
 \[\Psi(h_Z)=\int_MI(\tilde Z,d\tilde Z)\,\vol_{g}=\int_G\hat I(x)\circ\hat{\mathbf{Z}}(x)\,dx,\]
 where $\hat I\in C^\infty(G,\Sym^3(\m\oplus\Lambda^2\m)^\ast)^K$ and $\hat{\mathbf{Z}}\in C^\infty(G,\m\oplus\Lambda^2\m)^K$ are the functions on $G$ associated to the sections $I$ and $(\tilde Z,d\tilde Z)$, respectively.
 By \eqref{eq:killingcorresp}, under the inclusion
 \[\iota:\quad\g=\m\oplus\k\hookrightarrow\m\oplus\Lambda^2\m,\qquad Z\mapsto(Z_\m,\ad(Z_\k)),\]
 we can view $\hat{\mathbf{Z}}$ as the image of a function in $C^\infty(G,\g)^G$, which under the usual identification $C^\infty(G,\g)\cong\X(G)$ corresponds exactly to the \emph{right-invariant} vector field on $G$ generated by $Z$; namely
 \[\hat{\mathbf{Z}}(x)=\iota(\Ad(x^{-1})Z),\qquad x\in G.\]
 The section $I$, consisting of contractions of the $G$-invariant tensor $\sigma$, is itself $G$-invariant, meaning that $\hat I\in(\Sym^3(\m\oplus\Lambda^2\m)^\ast)^K$ is a constant function on $G$. By combining \eqref{eq:psiexpr}--\eqref{eq:FNbracket}, a quick calculation shows that
 \[\hat I\circ \iota = 3R-EQ\]
 (as polynomial functions), with $Q,R\in(\Sym^3\g^\ast)^K$ as defined previously. This finishes the proof.
\end{proof}

Up to a volume factor, this amounts to ``averaging'' the cubic polynomial $2EQ-6R$ to produce a $G$-invariant polynomial on $\g$. This averaging operation is nothing but orthogonal projection of the integrand to $(\Sym^3\g^\ast)^G$ with respect to any $G$-invariant inner product on $\Sym^3\g^\ast$.

To proceed, we examine the space $(\Sym^3\g^\ast)^G$ of invariant cubic forms. In the cases (i), (iii) of Prop.~\ref{symmetricied}, we have $G=\SU(m)$ and $\dim(\Sym^3\g^\ast)^G=1$, as already mentioned in \ref{sec:suncubic}. Therefore $\Psi$ must be proportional to the invariant $P$. We obtain the following corollary:

\begin{kor}
\label{orthtoP0}
 On the spaces $\SU(n)/\SO(n)$ and $\SU(2n)/\Sp(n)$, $n\geq3$, the Koiso obstruction $\Psi$ vanishes identically if and only if
 \[\langle 3R-EQ,P\rangle=0.\]
\end{kor}

By Proposition~\ref{critical}, integrability to second order is governed by the polynomial $\Psi$. If $\Psi$ vanishes identically, then clearly every IED is integrable to second order. If $\g=\su(m)$ and $\Psi\neq0$, i.e.~it is a nonzero multiple of $P$, the set of IED which are integrable to second order is the critical locus of $P$, which is described in \cite[Prop.~6.4]{OddGrArxiv}. For odd $m$, this set is just the origin, meaning that no IED is integrable, and thus the Einstein metric is rigid.

\begin{bem}
In \cite{BHMW} it was noted that $\dim(\Sym^3\g^\ast)^G=2$ for case (ii), the Lie group $\SU(n)$. The authors showed that $\Psi$ is nonzero and explicitly gave its critical locus.

In case (iv) ($\rmE_6/\rmF_4$), the space  $(\Sym^3\g^\ast)^G$ is trivial. This means that $\Psi$ automatically vanishes, hence every infinitesimal Einstein deformation is integrable to second order.
\end{bem}

Returning to cases (i) and (iii), we take a closer look at the space $(\Sym^3\g^\ast)^K$ that the integrand lives in:

\begin{lem}
\label{Sym3gK}
 Let $\g=\su(m)$, and either $\k=\so(n)$, $n=m$, or $\k=\sp(n)$, $m=2n$. Then
 \[(\Sym^3\g^\ast)^K=(\Sym^3\m^\ast)^K\oplus(\m^\ast\otimes\Sym^2\k^\ast)^K,\]
the spaces $(\Sym^3\m^\ast)^K$ and $(\m^\ast\otimes\Sym^2\k^\ast)^K$ are one-dimensional, and spanned by the respective restrictions
 \[P_{\m\m\m}\in(\Sym^3\m^\ast)^K,\qquad P_{\k\k\m}\in(\m^\ast\otimes\Sym^2\k^\ast)^K\]
 of the cubic form $P=P_{\m\m\m}+P_{\k\k\m}\in(\Sym^3\g^\ast)^G$.
\end{lem}

\begin{kor}
\label{constants}
 There exists constants $\kappa,\lambda_i,\lambda\in\R$ such that $Q=\kappa P_{\m\m\m}$, $R_i=\lambda_i P_{\k\k\m}$ for $i=1,\ldots,4$, and $R=\lambda P_{\k\k\m}$.
\end{kor}

Because the Cartan decomposition $\g=\k\oplus\m$ is orthogonal for our chosen inner product on $\g$, the summands $\Sym^3\m^\ast$ and $\m^\ast\otimes\Sym^2\k^\ast$ will be orthogonal for the $G$-invariant inner product on $\Sym^3\g^\ast$ given in \eqref{eq:inprod}; in particular $P_{\m\m\m}\perp P_{\k\k\m}$. Combining with Lemma~\ref{Sym3gK} and Corollary~\ref{constants}, we can rewrite the condition in Corollary~\ref{orthtoP0}:

\begin{kor}
\label{obstrequiv}
 On the spaces $\SU(n)/\SO(n)$ and $\SU(2n)/\Sp(n)$, $n\geq3$, the Koiso obstruction $\Psi$ vanishes identically if and only if
 \[E\kappa|P_{\m\m\m}|^2-3\lambda|P_{\k\k\m}|^2=0.\]
\end{kor}

It remains to determine the quantities $|P_{\m\m\m}|^2,|P_{\k\k\m}|^2,\kappa$, and $\lambda$ for each of the spaces $\SU(n)/\SO(n)$ and $\SU(2n)/\Sp(n)$. It turns out that these calculations are simplified by the use of the sandwich operators introduced in \ref{sec:sandwich}.

\section{Norms and ratios of polynomials}
\label{sec:Section5}
\subsection{Relevant constants}

Let us record a few constants related to the spaces $\SU(n)/\SO(n)$ and $\SU(2n)/\Sp(n)$. We handle both cases at the same time. From here on, let $\g=\su(m)$ with either $m=n$, $\k=\so(n)$, $\m=\i\Herm_0(n,\R)$ or $m=2n$, $\k=\sp(n)$, $\m=\i\Herm_0(n,\H)$. On $\g$ and its subspaces, the implied inner product is always minus the trace form, which by \eqref{eq:killingformsu} is $-\frac{1}{2m}B_\g$. It follows then from \cite[Prop.~7.93]{besse} that the Einstein constant of $(M,g)$ is $E=m$.

The other relevant Killing forms are related to the trace form by
\begin{align}
  B_{\so(n)}&=(n-2)\tr,&B_{\sp(n)}&=2(n+1)\tr.
  \label{eq:killingforms}
\end{align}

\begin{lem}
\label{cask}
The quadratic Casimir constants of $\k$ on the defining representation and on $\g=\k\oplus\m$ are given as follows.
\begin{enumerate}[\upshape(i)]
 \item For $\k=\so(n)$:
 \[\Cas^\k_{\R^n}=\frac{n-1}{2},\qquad\Cas^\k_\k=n-2,\qquad\Cas^\k_\m=n.\]
 \item For $\k=\sp(n)$:
 \[\Cas^\k_{\C^{2n}}=\frac{2n+1}{2},\qquad\Cas^\k_\k=2(n+1),\qquad\Cas^\k_\m=2n.\]
\end{enumerate}
\end{lem}
\begin{proof}
 For any simple Lie algebra $\h$, the Casimir eigenvalue of the adjoint representation is $1$ when the inner product on $\h$ is given by $-B_\h$. Thus, with trace form inner product, we find
 \[\Cas^{\so(n)}_{\so(n)}=n-2,\qquad\Cas^{\sp(n)}_{\sp(n)}=2(n+1).\]
 The Casimir eigenvalues of the other representations can be calculated using Freudenthal's formula \eqref{eq:freudenthal}. Let $(\omega_1,\omega_2,\ldots)$ denote a basis of fundamental weights of $\k$, with ordering given by Bourbaki's convention. Then
 \begin{align*}
  \R^n&=V_{\omega_1},&\m&\cong\Sym^2_0\R^n=V_{2\omega_1}&\text{for }\k=\so(n),\\
  \C^{2n}&=V_{\omega_1},&\m&\cong\Lambda^2_0\C^{2n}=V_{\omega_2}&\text{for }\k=\sp(n).
 \end{align*}
 Proceeding as in \cite[\S3]{SW22}, we obtain the desired eigenvalues.
\end{proof}

Combined with Lemma~\ref{cassun} we obtain another important constant:

\begin{kor}
\label{casm}
The \emph{partial} Casimir constant $\Cas^\m_{\C^m}:=\Cas^\g_{\C^m}-\Cas^\k_{\C^m}$ is given by
 \begin{align*}
  \Cas^\m_{\C^m}&=\frac{(n+2)(n-1)}{2n}&\text{for}\quad\k&=\so(n),\\
  \Cas^\m_{\C^m}&=\frac{(2 n + 1) (n - 1)}{2 n}&\text{for}\quad\k&=\sp(n).
 \end{align*}
\end{kor}

Note also that $\Cas^\k_{\C^m}=\frac{\dim\k}{m}$ and $\Cas^\m_{\C^m}=\frac{\dim\m}{m}$ which one can show directly using the trace form on $\C^m$, without having to appeal to Freudenthal's formula, as in the proof of Lemma~\ref{cassun}.

In light of the inclusion $\g\subset\End\C^m$, with $\C^m$ being either the $\so(n)$-module $(\R^n)^\C$ or the $\sp(n)$-module $\C^{2n}$, we can talk about the sandwich constants of $\k$ on $\g$. For convenience, we define the \emph{partial sandwich operators} of the complement $\m$ as
\[\SW^\m_{\End V}:=\SW^\g_{\End V}-\SW^\k_{\End V}\]
for any representation $V$ of $\g$, and denote its eigenvalue on a $\k$-irreducible submodule $W\subset\End V$ by $\SW^\m_W$. Inside the ring of matrices $\End\C^m$ we have
\begin{equation}
\Cas^\m_{\C^m}=-\sum_iE_i^2,\qquad \SW^\m_{\End\C^m}(A)=\sum_iE_iAE_i
\label{partialbase}
\end{equation}
for $A\in\End\C^m$ and an orthonormal basis $(E_i)$ of $\m$.

A straightforward computation using Lemma~\ref{cask} and Corollary~\ref{sandwichonirr} yields:

\begin{lem}
\label{sandwichconst}
 The (partial) sandwich constants of $\k$ and its complement $\m$ on the representation $\g=\k\oplus\m\subset\End\C^m$ are as follows.
 \begin{enumerate}[\upshape(i)]
 \item For $\k=\so(n)$:
 \[\SW^\k_\k=-\frac12,\qquad \SW^\k_\m=\frac12,\qquad \SW^\m_\k=\frac{n+2}{2n},\qquad \SW^\m_\m=\frac{2-n}{2n}.\]
 \item For $\k=\sp(n)$:
 \[\SW^\k_\k=\frac12,\qquad \SW^\k_\m=-\frac12,\qquad \SW^\m_\k=\frac{1-n}{2n},\qquad \SW^\m_\m=\frac{n+1}{2n}.\]
\end{enumerate}
\end{lem}

It is an elementary fact that the anticommutator of (skew-)Hermitian matrices is Hermitian. Therefore with $\m=\i\Herm_0(n,\AA)$ and $\k=\su(n,\AA)$ we have
\begin{align}
 \{\m,\m\}&\subset\i\m\oplus\R\Id,\label{mmanticomm}\\
 \{\k,\k\}&\subset\i\m\oplus\R\Id,\label{kkanticomm}
\intertext{and similarly one can check that}
 \{\m,\k\}&\subset\i\k.\label{mkanticomm}
\end{align}
We shall soon calculate contractions involving trace form and anticommutator. For this, we note the formula for the orthogonal projection of an anticommutator to its trace-free part:
\begin{equation}
 \{A,B\}_0=\{A,B\}-\frac{2}{m}\tr(AB)\Id,\qquad A,B\in\End\C^m.
 \label{tracefree}
\end{equation}

\subsection{The norms of $P_{\m\m\m}$ and $P_{\k\k\m}$}

We can now express the norms of the polynomials $P_{\m\m\m},P_{\k\k\m}\in(\Sym^3\g^\ast)^K$, defined in Lemma~\ref{Sym3gK}, in terms of Casimir and sandwich constants. To that end, we note the expansion
\[P(Z)=2\i\tr(Z^3)=2\i\tr(Z_\m^3+3Z_\m^2Z_\k^2+3Z_\m Z_\k^2+Z_\k^3),\]
for any $Z\in\g$, so with Lemma~\ref{Sym3gK} it follows that
\begin{equation}
P_{\m\m\m}(Z)=2\i\tr(Z_\m^3),\qquad P_{\k\k\m}(Z)=6\i\tr(Z_\m Z_\k^2),
\label{eq:P12explicit}
\end{equation}
or in polarized form
\begin{align}
P_{\m\m\m}(X,Y,Z)&=\i\tr(\{X_\m,Y_\m\}Z_\m),\label{eq:P1pol}\\ P_{\k\k\m}(X,Y,Z)&=\i\tr(\{X_\k,Y_\k\}Z_\m)+\i\tr(\{Z_\k,X_\k\}Y_\m)+\i\tr(\{Y_\k,Z_\k\}X_\m).
\label{eq:P2pol}
\end{align}

\begin{lem}
\label{p1norm}
 $|P_{\m\m\m}|^2=2m\Cas^\m_{\C^m}(-\SW^\m_\m+\Cas^\m_{\C^m}-\frac{2}{m})$.
\end{lem}
\begin{proof}
 With definition \eqref{eq:inprod} of the norm and \eqref{eq:P1pol}, we have
 \begin{align*}
  |P_{\m\m\m}|^2&=\sum_{i,j,k}P_{\m\m\m}(E_i,E_j,E_k)^2=-\sum_{i,j,k}\tr(\{E_i,E_j\}E_k)^2.
 \end{align*}
 By \eqref{mmanticomm}, $\{E_i,E_j\}\in\i\m\oplus\R\Id$. Thus, when we contract over the orthonormal basis $(E_k)$ and apply \eqref{tracefree}, we obtain
 \begin{align*}
  |P_{\m\m\m}|^2&=\sum_{i,j}\tr(\{E_i,E_j\}\{E_i,E_j\}_0)=\sum_{i,j}\tr\left(\{E_i,E_j\}^2-\tfrac{2}{m}\tr(E_iE_j)\{E_i,E_j\}\right).
 \end{align*}
 Multiplying out the anticommutators, using the cyclicity of the trace, and contracting over $(E_j)$ in the second term, this reduces to
 \begin{align*}
  |P_{\m\m\m}|^2&=\sum_{i,j}\tr(2E_iE_jE_iE_j+2E_i^2E_j^2)+\frac{2}{m}\sum_i\tr(2E_i^2).
 \end{align*}
 Using the partial sandwich and Casimir constants \eqref{partialbase}, this gives
 \begin{align*}
  |P_{\m\m\m}|^2&=2\left(\SW^\m_\m-\Cas^\m_{\C^m}+\tfrac{2}{m}\right)\sum_i\tr(E_i^2),
 \end{align*}
 and with $-\sum_i\tr(E_i^2)=\dim\m=m\Cas^\m_{\C^m}$ the assertion follows.
\end{proof}

\begin{lem}
\label{p2norm}
 $|P_{\k\k\m}|^2=6m\Cas^\k_{\C^m}(-\SW^\k_\k+\Cas^\k_{\C^m}-\frac{2}{m})$.
\end{lem}
\begin{proof}
 Let $(Y_i)$ denote an orthonormal basis of $\k$. Using definition \eqref{eq:inprod} and  \eqref{eq:P2pol}, we may write
 \begin{align*}
  |P_{\k\k\m}|^2&=3\sum_{i,j,k}P_{\k\k\m}(Y_i,Y_j,E_k)^2=-3\sum_{i,j,k}\tr(\{Y_i,Y_j\}E_k)^2.
 \end{align*}
 Since $\{Y_i,Y_j\}\in\i\m\oplus\R\Id$ be \eqref{kkanticomm}, we can use \eqref{tracefree} to obtain
 \begin{align*}
  |P_{\k\k\m}|^2&=3\sum_{i,j}\tr(\{Y_i,Y_j\}\{Y_i,Y_j\}_0),
 \end{align*}
 and analogous to the proof of Lemma~\ref{p1norm}, this reduces to
 \begin{align*}
  |P_{\m\m\m}|^2&=3\sum_{i,j}\tr(2Y_iY_jY_iY_j+2Y_i^2Y_j^2)+\frac{6}{m}\sum_i\tr(2Y_i^2)\\
  &=6\left(\SW^\k_\k-\Cas^\k_{\C^m}+\tfrac{2}{m}\right)\sum_i\tr(Y_i^2),
 \end{align*}
 and with $-\sum_i\tr(Y_i^2)=\dim\k=m\Cas^\k_{\C^m}$ the proof is finished.
\end{proof}

Finally, we may plug in the Casimir and sandwich constants from Lemmas \ref{cask} and \ref{sandwichconst} and Corollary~\ref{casm} to obtain explicit expressions for the norms.

\begin{kor}
 \label{pnorms}
 The norms $|P_{\m\m\m}|^2$ and $|P_{\k\k\m}|^2$ are given as follows.
 \begin{enumerate}[\upshape(i)]
  \item For $\k=\so(n)$:
  \[|P_{\m\m\m}|^2=\frac{(n+4)(n+2)(n-1)(n-2)}{2n},\qquad|P_{\k\k\m}|^2=\frac{3}{2}(n+2)(n-1)(n-2).\]
  \item For $\k=\sp(n)$:
  \[|P_{\m\m\m}|^2=\frac{2(2n+1)(n+1)(n-1)(n-2)}{n},\qquad|P_{\k\k\m}|^2=6(2n+1)(n+1)(n-1).\]
 \end{enumerate}
\end{kor}

As a sanity check, one may verify that indeed $|P_{\m\m\m}|^2+|P_{\k\k\m}|^2=|P|^2$, with the latter given in Lemma~\ref{p0norm}.

\subsection{Constants of proportionality}

It remains to work out the constants $\kappa,\lambda$ from Corollary~\ref{constants}. First, we formulate two auxiliary results:

\begin{lem}
\label{aux}
 For any $Z\in\g$ we have
 \begin{align}
  \sum_i\tr(\{Z_\m,\{Z_\m,\{Z_\m,E_i\}\}\}E_i)&=2(3\SW^\m_\m-\Cas^\m_{\C^m})\tr(Z_\m^3),\label{aux1}\\
  \sum_i\tr(\{Z_\m,E_i\}[Z_\k,[Z_\k,E_i]])&=-2(\Cas^\m_{\C^m}+2\SW^\m_\k-\SW^\m_\m)\tr(Z_\m Z_\k^2).\label{aux2}
 \end{align}
\end{lem}
\begin{proof}
 Multiplying out the anticommutators and using cyclicity of the trace, we can rewrite the above as
 \begin{align*}
  \sum_i\tr(\{Z_\m,\{Z_\m,\{Z_\m,E_i\}\}\}E_i)&=\sum_i\tr(2Z_\m^3E_i^2+6Z_\m^2E_iZ_\m E_i),\\
  \sum_i\tr(\{Z_\m,E_i\}[Z_\k,[Z_\k,E_i]])&=\sum_i\tr(\{Z_\m,Z_\k^2\}E_i^2-2\{Z_\m,Z_\k\}E_iZ_\k E_i+2Z_\k^2E_iZ_\m E_i).
 \end{align*}
 Employing the partial sandwich and Casimir constants \eqref{partialbase}, we arrive at the desired identities.
\end{proof}

Now we are ready calculate $\kappa$.

\begin{lem}
\label{kappa}
 \begin{align*}
  \kappa&=\Cas^\m_{\C^m}-3\SW^\m_\m-\frac{8}{m}=
  \begin{cases}
   (n^2+4n-24)/(2n)&\text{for}\quad\k=\so(n),\\
   (n^2-2n-6)/n&\text{for}\quad\k=\sp(n).
  \end{cases}
 \end{align*}
\end{lem}
\begin{proof}
 Let $Z\in\g$. By definition,
 \begin{align*}
  Q(Z)&=-\i\sum_{i,j,k}\tr(\{Z_\m,E_i\}E_j)\tr(\{Z_\m,E_i\}E_k)\tr(\{Z_\m,E_j\}E_k).
 \end{align*}
 By \eqref{mmanticomm}, $\{Z,E_j\}\in\i\m\oplus\R\Id$. Thus, when we contract over the orthonormal basis $(E_k)$ and apply \eqref{tracefree}, we obtain
 \begin{align*}
  Q(Z)&=\i\sum_{i,j}\tr(\{Z_\m,E_i\}E_j)\tr(\{Z_\m,E_i\}\{Z_\m,E_j\}_0)\\
  &=\i\sum_{i,j}\tr(\{Z_\m,E_i\}E_j)\Big(\tr(\{Z_\m,E_i\}\{Z_\m,E_i\})-\tfrac2m\tr(\{Z_\m,E_i\})\tr(Z_\m E_j)\Big)\\
  &=\i\sum_{i,j}\tr(\{Z_\m,E_i\}E_j)\Big(\tr(\{Z_\m,\{Z_\m,E_i\}\}E_i)-\tfrac4m\tr(Z_\m E_i)\tr(Z_\m E_j)\Big),
 \end{align*}
 using that $\{Z_\m,\cdot\,\}$ is a symmetric endomorphism in the last step. Next we contract over $(E_j)$ using \eqref{mmanticomm} and \eqref{tracefree},
 \begin{align*}
  Q(Z)&=-\i\sum_i\tr(\{Z_\m,\{Z_\m,E_i\}\}\{Z_\m,E_i\}_0)+\frac{4\i}{m}\sum_i\tr(\{Z_\m,E_i\}Z_\m)\tr(Z_\m E_i)\\
  &=-\i\sum_i\tr(\{Z_\m,\{Z_\m,E_i\}\}\{Z_\m,E_i\})+\frac{8\i}{m}\sum_i\tr(\{Z_\m,E_i\}Z_\m)\tr(Z_\m E_i)\\
  &=-\i\sum_i\tr(\{Z_\m,\{Z_\m,\{Z_\m,E_i\}\}\}E_i)-\frac{16\i}{m}\tr(Z_\m^3),
 \end{align*}
 and after an application of \eqref{aux1} and \eqref{eq:P12explicit}, we find
 \begin{align*}
  Q(Z)&=(\Cas^\m_{\C^m}-3\SW^\m_\m-\tfrac8m)P_{\m\m\m}(Z),
 \end{align*}
 which is the first assertion. The rest follows by applying Lemmas~\ref{cask} and \ref{sandwichconst} and Corollary~\ref{casm}.
\end{proof}

Before we calculate the constant $\lambda$, we identify a few relations between the polynomials $R_1,\ldots, R_4$ from Lemma~\ref{psipol}.

\begin{lem}
 $2R_1+R_2=0$ and $R_1=R_3+R_4$.
\end{lem}
\begin{proof}
 By the invariance of $\sigma\in(\Sym^3\m^\ast)^K$, we have
 \[\sigma([Z_\k,E_j],E_k,E_l)+\sigma([Z_\k,E_k],E_j,E_l)+\sigma([Z_\k,E_l],E_j,E_k)=0.\]
 Applying this to the expressions for $R_1$ and $R_2$ in Lemma~\ref{psipol} and using that
 \[\sum_{k,l}\sigma(\,\cdot\,,E_k,E_l)\sigma([Z_\k,E_k],E_j,E_l)=\sum_{k,l}\sigma(\,\cdot\,,E_k,E_l)\sigma([Z_\k,E_l],E_j,E_k),\]
 we find $2R_1+R_2=0$. Similarly, we have
 \[\sigma([Z_\k,E_i],E_k,E_l)=-\sigma(E_i,[Z_\k,E_k],E_l)-\sigma(E_i,E_k,[Z_\k,E_l]).\]
 Using the skew-symmetry of $\ad(Z_\k)$, we obtain
 \begin{align*}
  -\sum_{k,l}\sigma(E_i,[Z_\k,E_k],E_l)\sigma([Z_\k,E_k],E_j,E_l)&=\sigma(E_i,E_k,E_l)\sigma([Z_\k,[Z_\k,E_k]],E_j,E_l)\\
  &=\sigma(E_i,E_k,E_l)\sigma(E_j,E_k,[Z_\k,[Z_\k,E_l]]),\\
  -\sum_{k,l}\sigma(E_i,E_k,[Z_\k,E_l])\sigma([Z_\k,E_k],E_j,E_l)&=\sigma(E_i,E_k,E_l)\sigma(E_j,[Z_\k,E_k],[Z_\k,E_l]),
 \end{align*}
 which shows $R_1=R_3+R_4$.
\end{proof}

As a consequence, we find

\begin{kor}
\label{lambda13}
 $\lambda=8\lambda_1-4\lambda_3$.
\end{kor}

It remains then to just calculate $\lambda_1$ and $\lambda_3$.

\begin{lem}
\label{lambda1}
 \begin{align*}
  \lambda_1&=-\frac13(\Cas^\m_{\C^m}+\SW^\m_\k-\tfrac{4}{m})(\Cas^\m_{\C^m}+2\SW^\m_\k-\SW^\m_\m).
 \end{align*}
\end{lem}
\begin{proof}
 Fix $Z\in\g$. By definition,
 \begin{align*}
  R_1(Z)&=-\i\sum_{i,j,k,l}\tr(\{Z_\m,E_i\}E_j)\tr(\{[Z_\k,E_i],E_k\}E_l)\tr(\{[Z_\k,E_k],E_j\}E_l).
 \end{align*}
 By \eqref{mmanticomm}, $\{[Z_\k,E_i],E_k\}\in\i\m\oplus\R\Id$. Contracting over $(E_l)$ and applying \eqref{tracefree}, we obtain
 \begin{align*}
  R_1(Z)&=\i\sum_{i,j,k}\tr(\{Z_\m,E_i\}E_j)\tr(\{[Z_\k,E_i],E_k\}\{[Z_\k,E_k],E_j\}_0)\\
  &=\i\sum_{i,j,k}\tr(\{Z_\m,E_i\}E_j)\Big(\tr(\{[Z_\k,E_i],E_k\}\{[Z_\k,E_k],E_j\})\\
  &\qquad\qquad\qquad\qquad\qquad-\frac{2}{m}\tr(\{[Z_\k,E_i],E_k\})\tr([Z_\k,E_k]E_j)\Big).
 \end{align*}
 Multiplying this out, we find
 \begin{align*}
  R_1(Z)&=\i\sum_{i,j,k}\tr(\{Z_\m,E_i\}E_j)\tr\Big([Z_\k,E_i]\Big(-E_k^2AE_j+E_jAE_k^2+[E_kAE_k,E_j]\\
  &\qquad\qquad\qquad\qquad\qquad\qquad\qquad\qquad+[Z_\k,E_kE_jE_k]-E_k[Z_\k,E_j]E_k\Big)\Big)\\
  &\quad-\frac{4\i}{m}\sum_{i,j,k}\tr(\{Z_\m,E_i\}E_j)\tr([Z_\k,E_i]E_k)\tr([Z_\k,E_k]E_j).
 \end{align*}
 In the first sum, we use \eqref{partialbase}, while in the second sum, we contract over the basis $(E_k)$, noting that $[Z_\k,E_i]\in\m$:
 \begin{align*}
  R_1(Z)&=\i\sum_{i,j}\tr(\{Z_\m,E_i\}E_j)\tr((\Cas^\m_{\C^m}+\SW^\m_\k)[Z_\k,E_i][Z_\k,E_j])\\
  &\quad+\tfrac{4\i}{m}\sum_{i,j}\tr(\{Z_\m,E_i\}E_j)\tr([Z_\k,[Z_\k,E_i]]E_j).
 \end{align*}
 Using the skew-symmetry of $\ad(Z_\k)$, and contracting over $(E_j)$, this simplifies to
 \begin{align*}
  R_1(Z)&=\i\left(\Cas^\m_{\C^m}+\SW^\m_\k-\tfrac{4}{m}\right)\sum_{i}\tr(\{Z_\m,E_i\}[Z_\k,[Z_\k,E_i]]).
 \end{align*}
 With \eqref{aux2} and \eqref{eq:P12explicit} we obtain
 \begin{align*}
  R_1(Z)&=-\tfrac13\left(\Cas^\m_{\C^m}+\SW^\m_\k-\tfrac{4}{m}\right)(\Cas^\m_{\C^m}+2\SW^\m_\k-\SW^\m_\m)P_{\k\k\m}(Z),
 \end{align*}
 which is the desired statement.
\end{proof}

\begin{lem}
\label{lambda3}
 \begin{align*}
  \lambda_3&=-\frac13\left((\Cas^\m_{\C^m})^2+3\Cas^\m_{\C^m}\SW^\m_\k-3\Cas^\m_{\C^m}\SW^\m_\m-2(\SW^\m_\k)^2-\SW^\m_\k \SW^\m_\m+2(\SW^\m_\m)^2\right)\\
  &\quad+\frac{1}{3m}\left(6\Cas^\m_{\C^m}+16\SW^\m_\k-10\SW^\m_\m\right).
 \end{align*}
\end{lem}
\begin{proof}
 For any $Z\in\g$, by definition
 \begin{align*}
  R_3(Z)&=-\i\sum_{i,j,k,l}\tr(\{Z_\m,E_i\}E_j)\tr(\{E_i,E_l\}E_k)\tr(\{E_j,E_k\},[Z_\k,[Z_\k,E_l]]).
 \end{align*}
 We contract over $(E_k)$, noting that $\{E_i,E_l\}\in\i\m\oplus\R\Id$ by \eqref{mmanticomm}, and use \eqref{tracefree}:
 \begin{align*}
  R_3(Z)&=\i\sum_{i,j,l}\tr(\{Z_\m,E_i\}E_j)\tr(\{E_j,\{E_i,E_l\}_0\}[Z_\k,[Z_\k,E_l]])\\
  &=\i\sum_{i,j,l}\tr(\{Z_\m,E_i\}E_j)\Big(\tr(\{E_j,\{E_i,E_l\}\}[Z_\k,[Z_\k,E_l]])\\
  &\qquad\qquad\qquad\qquad\qquad\quad-\frac{2\i}{m}\tr(\{E_j,\Id\}[Z_\k,[Z_\k,E_l]])\tr(E_iE_l)\Big)
 \end{align*}
 Multiplying out, and using the symmetry in the indices $i\leftrightarrow j$, we obtain
 \begin{align*}
  R_3(Z)&=\i\sum_{i,j,l}\tr(\{Z_\m,E_i\}E_j)\tr(\{Z_\k^2,E_iE_j\}E_l^2-2\{Z_\k,E_iE_j\}E_lZ_\k E_l+2\{Z_\k^2,E_i\}E_lE_jE_l\\
  &\qquad\qquad\qquad\qquad\quad+2E_iE_jE_lZ_\k^2E_l-2Z_\k E_iE_l[E_j,Z_\k]E_l-2Z_\k E_iE_l\{E_j,Z_\k\}E_l)\\
  &\qquad-\frac{4\i}{m}\sum_{i,j,l}\tr(\{Z_\m,E_i\}E_j)\tr(E_j[Z_\k,[Z_\k,E_l]])\tr(E_iE_l).
 \end{align*}
 Noting that $[E_j,Z_\k]\in\m$, $Z_\k^2\in\i\m\oplus\R\Id$ by \eqref{kkanticomm} and $\{E_j,Z_\k\}\in\i\k$ by \eqref{mkanticomm}, we can use \eqref{partialbase} or contract $(E_l)$ to rewrite this as
 \begin{align*}
  R_3(Z)&=\i\sum_{i,j}\tr(\{Z_\m,E_i\}E_j)\tr(-\Cas^m_{\C^m}\{Z_\k^2,E_iE_j\}-2\SW^\m_\k\{Z_\k,E_iE_j\}Z_\k+2\SW^\m_\m\{Z_\k^2,E_i\}E_j\\
  &\qquad+2E_iE_j(\SW^\m_\m(Z_\k^2)_0-\tfrac{\tr(Z_\k^2)}{m}\Cas^\m_{\C^m}\Id)-2\SW^\m_\m Z_\k E_i[E_j,Z_\k]-2\SW^\m_\k Z_\k E_i\{E_j,Z_\k\})\\
  &\qquad+\frac{4\i}{m}\sum_{i,j}\tr(\{Z_\m,E_i\}E_j)\tr(E_j[Z_\k,[Z_\k,E_i]])\\
  &=\i\sum_{i,j}\tr(\{Z_\m,E_i\}E_j)\tr((-2\Cas^\m_{\C^m}+4\SW^\m_\m-6\SW^\m_\k)Z_\k^2E_iE_j+(2\SW^\m_\m-2\SW^\k_\k)Z_\k E_iZ_\k E_j\\
  &\qquad\qquad\qquad\qquad\qquad\quad-(\SW^\m_\m+\Cas^\m_{\C^m})\tfrac{2\tr(Z_\k^2)}{m}E_iE_j)\\
  &\qquad+\frac{4\i}{m}\sum_{i,j}\tr(\{Z_\m,E_i\}E_j)\tr(E_j[Z_\k,[Z_\k,E_i]]).\\
 \end{align*}
 Next, we contract over $(E_j)$ and use $\{Z_\m,E_i\}\in\i\m\oplus\R\Id$ together with
 \[\sum_i\tr(\{Z_\m,E_i\}E_i)=2\sum_i\tr(Z_\m E_i^2)=-2\Cas^\m_{\C^m}\tr Z_\m=0.\]
 This yields
 \begin{align*}
  R_3(Z)&=-\i\sum_i\tr((-2\Cas^\m_{\C^m}+4\SW^\m_\m-6\SW^\m_\k)Z_\k^2E_i\{Z_\m,E_i\}_0\\
  &\qquad\qquad\qquad\qquad+(2\SW^\m_\m-2\SW^\k_\k)Z_\k E_iZ_\k\{Z_\m,E_i\}_0)\\
  &\qquad-\frac{4\i}{m}\sum_i\tr(\{Z_\m,E_i\}[Z_\k,[Z_\k,E_i]]).
 \end{align*}
 For the individual terms, we use \eqref{tracefree} and \eqref{partialbase}:
 \begin{align*}
  \sum_i\tr(Z_\k^2E_i\{Z_\m,E_i\}_0)&=\sum_i\tr(Z_\k^2E_iZ_\m E_i+Z_\k^2E_i^2Z_\m-\tfrac{2}{m}\tr(Z_\m E_i)Z_\k^2E_i)\\
  &=(-\Cas^\m_{\C^m}+\SW^\m_\m+\tfrac{2}{m})\tr(Z_\m Z_\k^2),\\
  \sum_i\tr(Z_\k E_iZ_\k\{Z_\m,E_i\}_0)&=\sum_i\tr(Z_\k Z_\m E_iZ_\k E_i+Z_\k E_iZ_\k E_iZ_\m-\tfrac{2}{m}\tr(Z_\m E_i)Z_\k E_iZ_\k)\\
  &=(2S^\m_\k+\tfrac{2}{m})\tr(Z_\m Z_\k^2)
 \end{align*}
 Together with these and \eqref{aux2}, we finally obtain
 \begin{align*}
  R_3(Z)&=\Big((-2\Cas^\m_{\C^m}+4\SW^\m_\m-6\SW^\m_\k)(-\Cas^\m_{\C^m}+\SW^\m_\m+\tfrac{2}{m})+(2\SW^\m_\m-2\SW^\k_\k)(2S^\m_\k+\tfrac{2}{m})\\
  &\qquad\qquad-\tfrac{8}{m}(\Cas^\m_{\C^m}+2\SW^\m_\k-\SW^\m_\m)\Big)(-\i)\tr(Z_\m Z_\k^2)
 \end{align*}
 Expanding and applying \eqref{eq:P12explicit} finishes the proof.
\end{proof}

Combining Lemmas~\ref{lambda13}, \ref{lambda1} and \ref{lambda3}, as well as substituting in the Casimir and sandwich constants from Lemmas~\ref{cask} and \ref{sandwichconst} and Corollary~\ref{casm}, we obtain the following expressions.

\begin{kor}
\label{lambda}
 \begin{align*}
  \lambda&=-\frac13(4(\Cas^\m_{\C^m})^2+12\Cas^\m_{\C^m}\SW^\m_\k+4\Cas^\m_{\C^m}\SW^\m_\m+24(\SW^\m_\k)^2-4\SW^\m_\k \SW^\m_\m-8(\SW^\m_\m)^2)\\
  &\quad+\frac{8}{3m}(\Cas^\m_{\C^m}+\SW^\m_\m)\\
  &=\begin{cases}
   -(n+4)(n^2+8)/(3n)&\text{for}\quad\k=\so(n),\\
   -4(n-2)(n^2+2)/(3n)&\text{for}\quad\k=\sp(n).
  \end{cases}
 \end{align*}
\end{kor}

According to Corollary~\ref{orthtoP0}, we may now combine Corollary~\ref{pnorms}, Lemma~\ref{kappa} and Corollary~\ref{lambda} to assemble the obstruction integral.

\begin{kor}
\label{psiandp}
Koiso's obstruction integral is given by the following multiple of $P$.
 \begin{align*}
  \Psi&=\frac{\langle 3R-EQ,P\rangle}{|P|^2}P=\frac{-E\kappa|P_{\m\m\m}|^2+3\lambda|P_{\k\k\m}|^2}{|P|^2}P\\
  &=\begin{cases}
   -\dfrac{(n+4)(7n^2+4n+24)}{8(n+1)}P&\text{for}\quad\k=\so(n),\\
   -\dfrac{(n-2)(7n^2-2n+6)}{(2n-1)}P&\text{for}\quad\k=\sp(n).
  \end{cases}
 \end{align*}
\end{kor}

In particular, $\Psi$ does not vanish identically provided $n\geq3$.

\section{Rigidity and nonlinear instability}
\label{sec:Section6}

We are now ready to prove our main result.

\begin{thm}
\label{main}
 Let $n\geq3$, and let $(M=G/K,g)$ be one of the Riemannian symmetric spaces $\SU(n)/\SO(n)$ or $\SU(2n)/\Sp(n)$.
 \begin{enumerate}[\upshape(i)]
  \item The set of infinitesimal Einstein deformations in $\varepsilon(g)$ which are integrable to second order corresponds to the variety
  \[\quadric=\left\{X\in\g\,\middle|\,X^2=\frac{\tr(X^2)}{m}I_m\right\}\]
  under the $G$-equivariant isomorphism $\varepsilon(g)\cong\g=\su(m)$, where $m=n$ or $m=2n$.
  \item If $n$ is odd, all infinitesimal Einstein deformations of $\SU(n)/\SO(n)$ are obstructed to second order, and hence the metric $g$ is rigid.
 \end{enumerate}
\end{thm}
\begin{proof}
 We recollect the strategy so far. Let $h=h_Z\in\varepsilon(g)$ be the IED associated to an element $Z\in\g$ via the parallel $3$-tensor $\sigma$. Proposition~\ref{critical} states that $h$ is integrable to second order if and only if $h_Z$ is a critical point of the obstruction integral $\Psi$, which we know a priori to correspond to a multiple of the cubic form $P$ from \eqref{eq:suncubic} since $\dim(\Sym^3\g^\ast)^G=1$. By virtue of Corollary~\ref{psiandp}, $\Psi$ is never identically zero if $n\geq3$, and thus the critical points of $\Psi$ correspond to the critical points of $P$ under the identification $Z\mapsto h_Z$. The variety of critical points is precisely the set $\quadric$ above, and it consists of just the origin for $m$ odd \cite[Prop.~6.4]{OddGrArxiv}. Thus, in the case of odd $m$, none of the IED are integrable, which implies that the Einstein metric is rigid.
\end{proof}

Recall that an Einstein metric $g$ on a compact, oriented manifold $M$ is called \emph{stable for the Einstein--Hilbert action}, or just \emph{$\EH$-stable}, if it is a local maximum of $\EH$ restricted to the manifold $\CS_g$ of constant scalar curvature metrics with the same total volume as $g$. A necessary condition for this is \emph{linear semistability}, that is
\[\EH_g''(h,h)\leq 0\qquad\text{for all }h\in\TT(M),\]
but this condition is not sufficient in general. As already observed by Kröncke \cite{dynamical}, the nonvanishing of the third variation $\EH_g'''$ on $\varepsilon(g)=\ker\EH_g''\big|_{\TT(M)}$ is a sufficient criterion for $\EH$-instability. This is closely related to Koiso's obstruction integral:

\begin{lem}
\label{thirdvar}
 Let $(M,g)$ be a closed, orientable Einstein manifold. For any $h\in\varepsilon(g)$,
 \[\EH_g'''(h,h,h)=-\Ltwoinprod{\Ric_g''(h,h)}{h}=-\frac12\Psi(h).\]
\end{lem}
\begin{proof}
 This is essentially folklore; an account is given in \cite[\S3]{KS24}. The calculation is routine and rests on the fact \cite[Prop.~4.17]{besse} that
 \[\EH_g'(h)=\Ltwoinprod{-\Ric_g+\frac{\scal_g}{2}g}{h}\]
 for arbitrary metrics $g$ and variations $h\in\Sy^2(M)$; moreover, any IED $h\in\varepsilon(g)$ of an Einstein metric $g$ is trace-free, i.e.~$\langle g,h\rangle=0$, and annihilates the first and second variations $\EH_g'$, $\EH_g''$.
\end{proof}

\begin{thm}
\label{thm:Sunstable}
 The Riemannian symmetric spaces $\SU(n)/\SO(n)$ and $\SU(2n)/\Sp(n)$, $n\geq3$, are unstable as critical points of the Einstein--Hilbert action.
\end{thm}
\begin{proof}
 From Corollary~\ref{psiandp} we know that in both cases, $\Psi$ is not identically zero; in particular, there exists $h\in\varepsilon(g)$ such that $\Psi(h)\neq0$. Now, since $g$ is Einstein and $h\in\varepsilon(g)$, we have $\EH'_g(h)=0$ and $\EH''_g(h,h)=0$, but $\EH_g'''(h,h,h)\neq0$ by Lemma~\ref{thirdvar}. Thus $g$ is not a local maximum of $\EH$ among constant scalar curvature metrics.
\end{proof}

\begin{bem}
\label{Sunstable2}
 For the bi-invariant metric on $\SU(n)$, $n\geq3$, instability for the Einstein--Hilbert action was shown in \cite{BHMW}. For the complex Grassmannians $\Gr_p(\C^{p+q})$, $p,q\geq2$, $\EH$-instability follows in the same fashion from \cite[Thm.~1.3]{OddGrArxiv}.
\end{bem}

The Einstein--Hilbert action $\EH$ is closely related to Perelman's $\nu$-entropy functional. Its critical points are precisely the shrinking gradient Ricci solitons, which include all Einstein metrics of positive scalar curvature as a subclass.

If $\bar g$ is a metric of constant scalar curvature and $d=\dim M$, then
\[\nu(\bar g)=\frac{2-d}{2}\log\Vol(M,\bar g)+\frac{d}{2}\log\EH(\bar g)+\frac{d}{2}(1-\log(2\pi d)),\]
see the proof of \cite[Thm.~8.3]{kroenckeeinstein}. This implies that at an Einstein metric $g$ of positive scalar curvature, which are critical points of both $\EH$ and $\nu$ restricted to fixed volume metrics, the second variations $\EH_g''$ and $\nu_g''$ coincide up to a (positive) factor on $\TT(M)\subset T_g\CS_g$. This has been utilized for example in the dynamical stability analysis for compact symmetric spaces \cite{caohe}.

Moreover it follows that the third variations $\EH_g'''$ and $\nu_g'''$ coincide up to a factor on $\varepsilon(g)=\ker\EH_g''\big|_{\TT(M)}$.

\begin{prop}[\cite{BHMW}, Thm.~D]
\label{thirdvarnu}
 Let $(M^n,g)$ be an Einstein manifold with Einstein constant $E>0$. For any $h\in\varepsilon(g)$,
 \[\nu'''_g(h,h,h)=-\frac{E^{d/2}}{4E(2\pi)^{d/2}}\Psi(h).\]
\end{prop}

Performing the same argument as above with the $\nu$-entropy and Proposition~\ref{thirdvarnu}, we obtain that the symmetric metrics on $\SU(n)/\SO(n)$ and $\SU(2n)/\Sp(n)$, $n\geq3$, are not local maxima of the $\nu$-entropy. By \cite[Thm.~1.3]{dynamical}, this implies that they are \emph{dynamically unstable} in the sense of Sesum \cite{sesum}. Thus we have reproven:

\begin{kor}[\cite{caohe}, Thm.~4.3]
 The Riemannian symmetric spaces $\SU(n)/\SO(n)$ and $\SU(2n)/\Sp(n)$, $n\geq3$, are dynamically unstable as Ricci solitons.
\end{kor}

\begin{bem}
 It had already been proven in \cite[Thm.~4.3]{caohe} that the spaces (i)--(iv) in Proposition~\ref{symmetricied} are \emph{linearly} unstable for the $\nu$-entropy, that is, $\nu''_g\not\leq0$ (and hence dynamically unstable). The linear instabilities arise from eigenfunctions of the Laplace--Beltrami operator and are independent of the behaviour of $\EH$; as mentioned above, all spaces in Koiso's list (i)--(v) are $\EH$-linearly semistable.

 That the complex Grassmannians $\Gr_p(\C^{p+q})$ are dynamically unstable under the Ricci flow had been shown by Hall--Murphy--Waldron for $1\leq p\neq q$ \cite{HMW}, the case $p=1$ (i.e.~$\CP^{q}$) going back to Kröncke \cite{kroenckeeinstein}. The (nonlinear) instabilities constructed in \cite{HMW,kroenckeeinstein} again originate from eigenfunctions of the Laplace--Beltrami operator, not from the Einstein--Hilbert action. This construction does not work for $p=q$ because a certain cubic form on $\g$ vanishes identically, cf.~\cite[\S6.2]{OddGrArxiv}. Thus the instability result for $p=q$ (Remark~\ref{Sunstable2}) is indeed completely new.

 The question of $\EH$-stability remains open only for the space $\rmE_6/\rmF_4$, where Koiso's obstruction integral $\Psi$ vanishes identically.
\end{bem}

\section*{Acknowledgments}

The second named author was supported by the Procope project no.~\textbf{48959TL}, by the ANR grant no.~\textbf{ANR-21-CE40-0017} and FAPESP project no.~\textbf{2024/08127-4}, the latter two under the BRIDGES collaboration.

The third named author acknowledges the support received by the Special Priority  Program \textbf{SPP 2026 Geometry at Infinity} funded by the Deutsche Forschungsgemeinschaft DFG, and was partly supported by the Procope project no.~\textbf{57650868} (Germany)/\allowbreak\textbf{48959TL} (France).

Finally, the authors would like to thank Klaus Kröncke and Gregor Weingart for helpful discussions.

\end{document}